\documentclass[english,11pt]{article}%
\usepackage[T1]{fontenc}
\usepackage[latin9]{inputenc}
\usepackage{amsmath}
\usepackage{amssymb}
\usepackage{amsfonts}
\usepackage{babel}
\usepackage{fancyhdr}
\usepackage{color}
\usepackage[toc]{appendix}
\usepackage{graphicx}
\usepackage{graphicx}%
\providecommand{\U}[1]{\protect\rule{.1in}{.1in}}
\makeatletter
\newtheorem{theorem}{Theorem}

\newtheorem{corollary}[theorem]{Corollary}

\newtheorem{lemma}[theorem]{Lemma}

\newtheorem{proposition}[theorem]{Proposition}

\newtheorem{preremark}[theorem]{Remark}
\newenvironment{remark}    {\begin{preremark}\rm}{\end{preremark}}
\numberwithin{equation}{section}
\numberwithin{example}{section}
\numberwithin{theorem}{section}

\newenvironment{proof}[1][Proof]{\noindent\textbf{#1.} }{\ \rule{0.5em}{0.5em}}
\DeclareMathOperator*{\argmin}{arg\,min}
\makeatother
\begin{document}

\title{A large deviations analysis of certain qualitative properties of parallel
tempering and infinite swapping algorithms}
\author{J. Doll\thanks{Department of Chemistry, Brown University, Providence, RI
02912. Research supported in part by the National Science Foundation
(DMS-1317199), and the Defense Advanced Research Projects Agency
(W911NF-15-2-0122).}, P. Dupuis\thanks{Division of Applied Mathematics, Brown
University, Providence, RI 02912. Research supported in part by the Department
of Energy (DE-SC0010539), the National Science Foundation (DMS-1317199), and
the Defense Advanced Research Projects Agency (W911NF-15-2-0122).} and P.
Nyquist\thanks{Division of Applied Mathematics, Brown University, Providence,
RI 02912. Research supported in part by National Science Foundation
(DMS-1317199).}}
\maketitle

\begin{abstract}
Parallel tempering, or replica exchange, is a popular method for simulating
complex systems. The idea is to run parallel simulations at different
temperatures, and at a given swap rate exchange configurations between the
parallel simulations. From the perspective of large deviations it is optimal
to let the swap rate tend to infinity and it is possible to construct a
corresponding simulation scheme, known as infinite swapping. In this paper we
propose a novel use of large deviations for empirical measures for a more
detailed analysis of the infinite swapping limit in the setting of continuous
time jump Markov processes. Using the large deviations rate function and
associated stochastic control problems we consider a diagnostic based on
temperature assignments, which can be easily computed during a simulation. We
show that the convergence of this diagnostic to its a priori known limit is a
necessary condition for the convergence of infinite swapping. The rate
function is also used to investigate the impact of asymmetries in the
underlying potential landscape, and where in the state space poor sampling is
most likely to occur.
\end{abstract}

\section{Introduction}

In prior work \cite{dupliupladol} we used the large deviation rate function to
study rate of convergence questions for parallel tempering (PT, also known as
replica exchange) type computational methods. The analysis suggested the
construction of a related method obtained in the limit as the attempted swap
rate tends to infinity, which was labeled the infinite swapping (INS)
algorithm, along with more generally implementable partial infinite swapping
(PINS) algorithms. In the present paper we apply the rate function again along
with related constructions from stochastic optimal control to analyze certain
qualitative properties of INS and, to a lesser extent, of PT algorithms.

There are two main results. The first is the theoretical analysis of a
diagnostic for good sampling that was introduced in \cite{dolplafreliudup,
doldup}. The empirical measure of particle/temperature associations (i.e., the
fraction of time that a given particle is assigned to a given dynamic) is a
quantity that is easy to record during a simulation. It was shown in the
references that under mild assumptions, for PT, PINS and INS the empirical
measure converges to the uniform distribution in the large time limit. We show
in this paper that if the empirical measure is not in a small neighborhood of
the uniform distribution, then with overwhelming probability the numerical
approximation provided by INS is not close to the target product measure.
Hence the particle/temperature empirical measure provides a convenient
diagnostic for good sampling. Such a diagnostic can be very useful when
applying Monte Carlo to problems involving rare events, where the
approximation will in general not converge in a gradual and predictable way,
and other more standard diagnostics (e.g., empirical variance) may suggest
convergence when it has not occurred. Of course temperature is only one of
several parameters that can (and have) been used to index the ensemble of
systems used in parallel replica methods, and we expect that the results
developed here for temperature could be generalized to these other parameters
as well \cite{eardee}.

The second use of large deviation theory is to study other qualitative
properties of INS. It is known that the large deviation rate function can tell
one not just the decay rates associated with rare events, but also the most
likely way that the event will occur. This property has found many uses in the
application of large deviations to problems involving small random
perturbations of deterministic systems, where it is used to indicate those
pathways a process is most likely to follow on the way to a rare outcome.
However, our use here is fundamentally different, in that we will use the
large deviation rate function for empirical measures to understand and explain
various properties of sampling schemes. We consider two such applications. The
first is to characterize the impact of symmetries in the energy landscape on
the convergence properties of INS (and also to some degree PT). Indeed, we
find that symmetry properties of the potential play a large role in the
performance of PT and INS schemes, and uncover some counter-intuitive
behavior. For example, we show that reducing energy barriers can actually slow
convergence when the reduction leads to a type of symmetry breaking. The
second application is to identify those parts of the state space where
\textquotedblleft errors\textquotedblright\ in the sampling have the greatest
impact on the overall performance of the algorithms. Such information could,
in principle, be used to design a sampling scheme which focuses computational
effort on accurate approximation of distributional properties in these
regions, thereby leading to better performance. Both of these applications
exploit the fact that we solve variational problems associated with the rate
function by converting them to ergodic (average cost per unit time) stochastic
optimal control problems, which are then solved numerically. To the best of
our knowledge this is the first instance where large deviations results have
been used in this way, allowing us to obtain explicit information on the
convergence properties of simulation schemes.

The organization of the paper is as follows. Section \ref{sec:Model}
introduces the INS algorithm for a general class of jump Markov processes. The
use of this particular type of process is motivated in part by the fact that
variational problems involving the rate function can be solved numerically.
The particle/temperature empirical measure is also introduced. Section
\ref{sec:LD} describes the joint large deviation properties of the INS
numerical approximation and the particle/temperature empirical measure. The
extraction of information from the rate function will require the solution to
certain variational problems, for which a conversion into equivalent ergodic
stochastic control problems is more convenient. These controls problems and
their properties are given in Section \ref{sec:controlProblem}. Although the
material of this section is known (at least in some form) in the literature on
stochastic control, we could not find it in the form we need, nor is it likely
to be familiar to those using PT and related methods. We then apply these
ideas to the analysis of the diagnostic in Section \ref{sec:PT}, and the
analysis of other qualitative properties in Section \ref{sec:asymmetry}.

\section{Infinite swapping for jump Markov processes}

\label{sec:Model} In this section we describe the infinite swapping algorithms
that are appropriate for discrete spaces. Specifically, we consider the
setting of a single temperature continuous-time Markov jump processes on a
finite state space $\mathcal{S}$ of size $|\mathcal{S}|=N\in\mathbb{N}$. For a
description of infinite swapping corresponding to diffusion models or jump
Markov processes on an uncountable state space see \cite{dupliupladol}. The
results we present will carry over to these settings but with a more involved
analysis. Also, as noted previously the discrete formulation makes the
numerical solution of examples much simpler.

The infinite swapping algorithms take values in the product space
$\mathcal{S}^{K}$, where $K\in\mathbb{N}$ is the number of temperatures under
consideration. For $K\in\mathbb{N}$, $x,y,z$ and so on are used to denote
generic elements of $\mathcal{S}$, and the boldface versions such as
$\mathbf{x}=(x_{1},x_{2},\dots,x_{K})$ denote elements of the product space.
For notational brevity we often consider what corresponds to a two temperature
model and typically only comment on the extension to multiple temperatures. It
will become clear that, from a mathematical viewpoint, the inclusion of more
temperatures is much a matter of bookkeeping. However, the use of many
temperatures eventually introduces a practical challenge when implementing
infinite swapping algorithms. This is addressed in \cite{dupliupladol} where
the partial infinite swapping (PINS) algorithm is introduced.

In the setting of Markov jump processes with two temperatures $\tau_{1}$ and
$\tau_{2}$, we describe the dynamics of the process by two rate matrices,
$\Gamma^{1}$ and $\Gamma^{2}$, each of size $N\times N$. One should think of
$\Gamma^{1}$ as the dynamics that go with $\tau_{1}$ and similarly for
$\Gamma^{2}$ and $\tau_{2}$. For two states $x\neq y$, $\Gamma_{x,y}^{i}\geq0$
represents the rate at which the process jumps from $x$ to $y$.

As an example of rate matrices $\Gamma^{1}$ and $\Gamma^{2}$ that correspond
to temperatures $\tau_{1}$ and $\tau_{2}$, we consider a particular form of
so-called Glauber dynamics \cite{str2}. To be precise, let $\mu$ be a Gibbs
measure defined in terms of a potential $V:\mathcal{S}\rightarrow\mathbb{R}$
and temperature $\tau$:
\[
\mu(x)=e^{-V(x)/\tau}/Z(\tau),\ x\in\mathcal{S},
\]
where $Z(\tau)$ is the normalizing constant associated with the temperature
$\tau$.

To define a type of Glauber dynamics with $\mu$ as invariant measure, let $A$
be an $N\times N$ matrix with entries $A_{x,y}\in\{0,1\}$ for all
$x,y\in\mathcal{S}$, with $A_{x,x}=0$. The matrix $A$ defines the
communicating class(es) of $\mathcal{S}$. Next, we define a rate matrix
$\Gamma$ in terms of the potential $V$ and the temperature $\tau$ by%

\[%
\begin{split}
\Gamma_{x,y}  &  \doteq\text{exp}\left\{  -\frac{1}{\tau}\left(
V(y)-V(x)\right)  ^{+}\right\}  A_{x,y},\ y\neq x,\\
\Gamma_{x,x}  &  \doteq-\sum_{y\neq x}\Gamma_{x,y}.
\end{split}
\]
Then $q(x)=-\Gamma_{x,x}$ represents the total rate out of state $x$ for a
chain with dynamics according to $\Gamma$. Note that this is but one
particular form of Glauber dynamics, and is sometimes referred to as
Metropolis dynamics; see \cite{str2} for other forms and details. For any such
Glauber dynamics, $\mu$ is an invariant distribution for the associated
continuous time Markov chain. Under ergodicity, the empirical measure for such
a chain will converge to $\mu$. We assume $\Gamma^{1}$ and $\Gamma^{2}$ are
ergodic henceforth.

\subsection{Infinite swapping limit}

For simplicity we continue to keep the discussion to two temperatures
$\tau_{1}$ and $\tau_{2}$. In accordance with the previous section, let
$\mu_{1}$ and $\mu_{2}$ denote the invariant distributions associated with the
given dynamics and the two rate matrices be denoted $\Gamma^{1}$ and
$\Gamma^{2}$, respectively. Let $\mu$ be the product distribution
\[
\mu=\mu_{1}\times\mu_{2}.
\]
The infinite swapping limit for this setting is described in
\cite{dupliupladol} and we review only the basics. Let $\mathbf{X}%
^{0}=\{(X_{1}^{0}(t),X_{2}^{0}(t)):t\geq0\}$ denote the Markov process with
independent components, each having dynamics according to the rate matrices
$\Gamma^{i}$, $i=1,2$. The embedded Markov chain $\bar{\mathbf{X}}^{0}$ has
probability transition kernel
\[
P\left(  \bar{\mathbf{X}}^{0}(j+1)=(y_{1},y_{2})|\bar{\mathbf{X}}%
^{0}(j)=(x_{1},x_{2})\right)  =%
\begin{cases}
\frac{\Gamma_{x_{1},y_{1}}^{1}}{q_{1}(x_{1})+q_{2}(x_{2})}, & y_{1}\neq
x_{1},\ y_{2}=x_{2},\\
\frac{\Gamma_{x_{2},y_{2}}^{2}}{q_{1}(x_{1})+q_{2}(x_{2})}, & y_{1}%
=x_{1},\ y_{2}\neq x_{2},\\
0, & \text{otherwise}.
\end{cases}
\]
The dynamics thus describe a process for which, when the current state is
$(x_{1},x_{2})$, the support of the jump distribution is $\{(y_{1}%
,x_{2}):y_{1}\in\mathcal{S}\}\cup\{(x_{1},y_{2}):y_{2}\in\mathcal{S}\}$. The
jump times of $\mathbf{X}^{0}$ are exponential random variables with jump
rates given by $q(x_{1},x_{2})=q_{1}(x_{1})+q_{2}(x_{2})$. This is
conveniently summarized by the infinitesimal generator $\mathcal{L}^{0}$
associated with the process $\mathbf{X}^{0}$:
\begin{align*}
\mathcal{L}^{0}f(x_{1},x_{2})  &  =\frac{1}{q(x_{1},x_{2})}\sum_{y_{1}\neq
x_{1}}\left[  f(y_{1},x_{2})-f(x_{1},x_{2})\right]  \Gamma_{x_{1},y_{1}}^{1}\\
&  \quad+\frac{1}{q(x_{1},x_{2})}\sum_{y_{2}\neq x_{2}}\left[  f(x_{1}%
,y_{2})-f(x_{1},x_{2})\right]  \Gamma_{x_{2},y_{2}}^{2}.
\end{align*}
Let $\eta_{T}^{0}$ denote the empirical measure of $\mathbf{X}^{0}$,
\[
\eta_{T}^{0}(\cdot)=\frac{1}{T}\int_{0}^{T}\delta_{\mathbf{X}^{0}(t)}%
(\cdot)dt.
\]
By ergodicity, with probability one $\eta_{T}^{0}$ converges weakly to $\mu$
in $\mathcal{P}(\mathcal{S}^{2})$, the space of probability measures on
$\mathcal{S}^{2}$, as $T\rightarrow\infty$ and the empirical measure $\eta
_{T}^{0}$ is used to approximate ergodic averages associated with $\mu$. In
particular, this provides a way to obtain estimates of integrals associated
with $\mu_{1}$, the invariant measure associated with the lower temperature
and often the distribution of interest in practice. However, the problem of
rare-event sampling may cause the rate of convergence for $\eta_{T}^{0}$ to
$\mu$ to be slow, resulting in inaccurate estimates, especially with respect
to the low temperature marginal.

Let $\mathbf{X}^{a}$ be the process that corresponds to swaps according to a
Metropolis-type rule and with rate $a\geq0$. More precisely, let $b$ be the
function
\begin{equation}
b(x_{1},x_{2})=1\wedge\frac{\mu(x_{2},x_{1})}{\mu(x_{1},x_{2})} \label{eq:g}%
\end{equation}
and let the infinitesimal generator $\mathcal{L}^{a}$ associated with
$\mathbf{X}^{a}$ be given by
\[
\mathcal{L}^{a}f(x_{1},x_{2})=\mathcal{L}^{0}f(x_{1},x_{2})+ab(x_{1}%
,x_{2})\left[  f(x_{2},x_{1})-f(x_{1},x_{2})\right]  .
\]
$\mathbf{X}^{a}$ is a continuous time version of the well-known parallel
tempering algorithm \cite{eardee,gey,sugoka,swewan}.

The process $\mathbf{X}^{a}$ has two different kinds of jumps. Suppose the
process is in $(x_{1},x_{2})$ and let $s_{1}$ and $s_{2}$ be two exponential
random variables with parameter $q(x_{1},x_{2})$ and $a$, respectively. Then
the process has a jump after $s=s_{1}\wedge s_{2}$ units of time. Note that
the time $s$ to the next jump is an exponential random variable with parameter
$q(x_{1},x_{2})+a$. If $s_{1}<s_{2}$ the jump is according to the dynamics
specified by $\Gamma^{1}$ and $\Gamma^{2}$ (the $\mathcal{L}^{0}$ part of the
generator). If instead $s_{2}<s_{1}$, with probability $b(x_{1},x_{2})$ the
process jumps to $(x_{2},x_{1})$ - the two particles switch locations - and
with remaining probability it stays in $(x_{1},x_{2})$ (a failed swap
attempt). In terms of the embedded Markov chain $\bar{\mathbf{X}}^{a}$ this
corresponds to a probability transition kernel
\begin{align*}
&  P\left(  \bar{\mathbf{X}}^{a}(j+1)=(y_{1},y_{2})|\bar{\mathbf{X}}%
^{a}(j)=(x_{1},x_{2})\right) \\
&  =\frac{q(x_{1},x_{2})}{a+q(x_{1},x_{2})}P\left(  \bar{\mathbf{X}}%
^{0}(j+1)=(y_{1},y_{2})|\bar{\mathbf{X}}^{0}(j)=(x_{1},x_{2})\right) \\
&  \quad+\frac{a}{a+q(x_{1},x_{2})}\left[  b(x_{1},x_{2})\delta_{(x_{2}%
,x_{1})}(y_{1},y_{2})+(1-b(x_{1},x_{2}))\delta_{(x_{1},x_{2})}(y_{1}%
,y_{2})\right]  .
\end{align*}
As previously mentioned the jump times of $\mathbf{X}^{a}$ are exponential
random variables with jump rates $q(x_{1},x_{2})+a$ (when in state
$(x_{1},x_{2})$). Although easy to check using detailed balance, it is
important to note that introducing the second type of jump - swaps of the
particle locations - does not change the invariant distribution; just as for
$\mathbf{X}^{0}$, $\mu$ is the invariant distribution of $\mathbf{X}^{a}$.

The infinite swapping process, first studied in \cite{dupliupladol}, is the
limit process that arises as $a\rightarrow\infty$. However, the processes
$\{\mathbf{X}^{a}\}$ cannot be tight as $a$ grows due to the discontinuities
introduced by the swapping of particle locations and there is no well-defined
limit process. Instead, a process associated with the limit $a\rightarrow
\infty$ can be obtained by considering a temperature swapped version of
$\mathbf{X}^{a}$, denoted $\mathbf{Y}^{a}$. We give only a brief description
of $\mathbf{Y}^{a}$ and the reader is referred to \cite{dupliupladol} for
details and a more thorough discussion.

Rather than attempting to swap particle locations, one can let dynamics of the
different components of the Markov process $\mathbf{X}^{0}$ change. To
describe this, suppose we append a process $Z^{a}$ with values in $\{1,2\}$,
and consider the Markov process $(\mathbf{Y}^{a},Z^{a})$ with jump rates given
$(\mathbf{Y}^{a},Z^{a})=(x_{1},x_{2},z)$ equal to
\[%
\begin{array}
[c]{cc}%
\Gamma_{x_{1},y_{1}}^{1} & \text{for }(x_{1},x_{2},1)\rightarrow(y_{1}%
,x_{2},1)\\
\Gamma_{x_{2},y_{2}}^{2} & \text{for }(x_{1},x_{2},1)\rightarrow(x_{1}%
,y_{2},1)\\
\Gamma_{x_{1},y_{1}}^{2} & \text{for }(x_{1},x_{2},2)\rightarrow(y_{1}%
,x_{2},2)\\
\Gamma_{x_{2},y_{2}}^{1} & \text{for }(x_{1},x_{2},2)\rightarrow(x_{1}%
,y_{2},2)\\
ab(x_{1},x_{2}) & \text{for }(x_{1},x_{2},1)\rightarrow(x_{1},x_{2},2)\\
ab(x_{2},x_{1}) & \text{for }(x_{1},x_{2},2)\rightarrow(x_{1},x_{2},1)
\end{array}
.
\]
With this process the particles do not change location when a swap is
successful. Instead, the dynamics are swapped as indicated by the value of
$Z^{a}$. To account for this change the empirical measure must be redefined,
and in fact one uses
\begin{equation}
\frac{1}{T}\int_{0}^{T}\left[  1_{\left\{  Z^{a}(t)=1\right\}  }%
\delta_{(\mathbf{Y}_{1}^{a}(t),\mathbf{Y}_{2}^{a}(t))}(\cdot)+1_{\left\{
Z^{a}(t)=2\right\}  }\delta_{(\mathbf{Y}_{2}^{a}(t),\mathbf{Y}_{1}^{a}%
(t))}(\cdot)\right]  dt \label{eqn:Ya-emp-meas}%
\end{equation}
in lieu of
\[
\frac{1}{T}\int_{0}^{T}\delta_{\mathbf{X}^{a}(t)}(\cdot)dt.
\]

Because of the change in bookkeeping, the process $\mathbf{Y}^{a}$ as well as
the replacement for the empirical measure have well defined limits in
distribution as $a\rightarrow\infty$ \cite{dupliupladol}. Let $\mathbf{Y}%
^{\infty}=\{(Y_{1}^{\infty}(t),Y_{2}^{\infty}(t)):t\geq0\}$ denote the limit
process, referred to as the infinite swapping limit or the infinite swapping
process. It follows from the dynamics of $\mathbf{Y}^{a}$ that $\mathbf{Y}%
^{\infty}$ is a pure jump Markov process with infinitesimal generator
\begin{equation}
\mathcal{L}^{\infty}f(x_{1},x_{2})=\sum_{(y_{1},y_{2})\in\mathcal{S}^{2}%
}\left[  f(y_{1},y_{2})-f(x_{1},x_{2})\right]  \Gamma_{(x_{1},x_{2}%
),(y_{1},y_{2})}^{\infty} \label{eq:generatorINS}%
\end{equation}
where the rate matrix $\Gamma^{\infty}$ is defined as
\begin{equation}
\Gamma_{(x_{1},x_{2}),(y_{1},y_{2})}^{\infty}=%
\begin{cases}
\rho(x_{1},x_{2})\Gamma_{x_{1},y_{1}}^{1}+\rho(x_{2},x_{1})\Gamma_{x_{1}%
,y_{1}}^{2}, & y_{1}\neq x_{1},y_{2}=x_{2},\\
\rho(x_{1},x_{2})\Gamma_{x_{2},y_{2}}^{2}+\rho(x_{2},x_{1})\Gamma_{x_{2}%
,y_{2}}^{1}, & y_{1}=x_{1},y_{2}\neq x_{2},\\
0, & \text{otherwise},
\end{cases}
\label{eq:defGamma}%
\end{equation}
where
\begin{equation}
\rho(x_{1},x_{2})=\frac{\mu(x_{1},x_{2})}{\mu(x_{1},x_{2})+\mu(x_{2},x_{1})}
\label{eqn:rho}%
\end{equation}
is the relative weight $\mu$ assigns to the permutation $(x_{1},x_{2})$. The
limit of (\ref{eqn:Ya-emp-meas}) in distribution is
\begin{equation}
\frac{1}{T}\int_{0}^{T}\left[  \rho(\mathbf{Y}_{1}^{\mathbb{\infty}%
}(t),\mathbf{Y}_{2}^{\mathbb{\infty}}(t))\delta_{(\mathbf{Y}_{1}%
^{\mathbb{\infty}}(t),\mathbf{Y}_{2}^{\mathbb{\infty}}(t))}(\cdot
)+\rho(\mathbf{Y}_{2}^{\mathbb{\infty}}(t),\mathbf{Y}_{1}^{\mathbb{\infty}%
}(t))\delta_{(\mathbf{Y}_{2}^{\mathbb{\infty}}(t),\mathbf{Y}_{1}%
^{\mathbb{\infty}}(t))}(\cdot)\right]  dt. \label{eqn:Yinf-emp-meas}%
\end{equation}
In the general case, $K\geq2$, $\rho$ is defined similarly, with the
denominator in (\ref{eqn:rho}) now being the sum over of $\mu(\mathbf{x}%
^{\sigma})$ for all permutations $\sigma\in\Sigma_{K}$, and the replacement
for the empirical measure using a sum over all such permutations. It is shown
in \cite{dupliupladol} that the infinite swapping approximation
(\ref{eqn:Ya-emp-meas}) converges to $\mu$ faster than the corresponding
parallel tempering scheme (i.e., the empirical measure of $\mathbf{X}^{a}$)
for any $a\in\lbrack0,\infty)$, when the corresponding large deviations rate
functions are used to measure the rate of convergence.

Similar to the pre-limit dynamics, $\Gamma_{(x_{1},x_{2}),(x_{1},x_{2}%
)}^{\infty}=-\sum_{(y_{1},y_{2})\in\mathcal{S}^{2}}\Gamma_{(x_{1}%
,x_{2}),(y_{1},y_{2})}^{\infty}$ and we define $q^{\infty}(x_{1}%
,x_{2})=-\Gamma_{(x_{1},x_{2}),(x_{1},x_{2})}^{\infty}$ . The interpretation
is that the total rate out of state $(x_{1},x_{2})$ for the infinite swapping
process is a mixture, according to the weights $\rho(x_{1},x_{2})$ and
$\rho(x_{2},x_{1})$, of the rates out of $(x_{1},x_{2})$ and $(x_{2},x_{1})$
for the original (uncoupled) process. The following symmetry properties will
come in handy later on. Although not obvious, they follow immediately from the
definitions of $\Gamma^{\infty}$ and $q^{\infty}$ and the proof is merely a
matter of bookkeeping.

\begin{lemma}
\label{lemma:symmetry} For any $\mathbf{x}, \mathbf{y} \in\mathcal{S}^{K}$ and
permutation $\sigma\in\Sigma_{K}$, $q^{\infty} (\mathbf{x}) = q^{\infty
}(\mathbf{x}^{\sigma}) $ and $\Gamma^{\infty} _{\mathbf{x}, \mathbf{y}} =
\Gamma^{\infty} _{\mathbf{x}^{\sigma}, \mathbf{y}^{\sigma}}$.
\end{lemma}

\begin{proof}
For notational simplicity we limit the proof to the case $K=2$. Start by
considering the claim for $\Gamma^{\infty}$. If $(x_{1},x_{2})$ and
$(y_{1},y_{2})$ are chosen such that $\Gamma_{(x_{1},x_{2}),(y_{1},y_{2}%
)}^{\infty}=0$ then $\Gamma_{(x_{2},x_{1}),(y_{2},y_{1})}^{\infty}=0$ as well
and the claim holds. Without loss of generality we can consider $y_{1}\neq
x_{1}$ and $y_{2}=x_{2}$, so that
\[
\Gamma_{(x_{1},x_{2}),(y_{1},y_{2})}^{\infty}=\rho(x_{1},x_{2})\Gamma
_{x_{1},y_{1}}^{1}+\rho(x_{2},x_{1})\Gamma_{x_{1},y_{1}}^{2}.
\]
Labeling the states $(u_{1},u_{2})=(x_{2},x_{1})$ and $(v_{1},v_{2}%
)=(y_{2},y_{1})$, we have $v_{1}=u_{1}$ and $v_{2}\neq u_{2}$. By
\eqref{eq:defGamma},
\begin{align*}
\Gamma_{(x_{2},x_{1}),(y_{2},y_{1})}^{\infty}  &  =\Gamma_{(u_{1}%
,u_{2}),(v_{1},v_{2})}^{\infty}\\
&  =\rho(u_{1},u_{2})\Gamma_{u_{2},v_{2}}^{2}+\rho(u_{2},u_{1})\Gamma
_{u_{2},v_{2}}^{1}\\
&  =\rho(x_{2},x_{1})\Gamma_{x_{1},y_{1}}^{2}+\rho(x_{1},x_{2})\Gamma
_{x_{1},y_{1}}^{1}\\
&  =\Gamma_{(x_{1},x_{2}),(y_{1},y_{2})}^{\infty}.
\end{align*}
The case where $y_{1}=x_{1},y_{2}\neq x_{2}$ is completely analogous. This
confirms the claim for $\Gamma^{\infty}$ and the symmetry property for
$q^{\infty}$ then follows directly from the definition:
\begin{align*}
q^{\infty}(x_{1},x_{2})  &  =\sum_{(y_{1},y_{2})\in\mathcal{S}^{2}}%
\Gamma_{(x_{1},x_{2}),(y_{1},y_{2})}^{\infty}\\
&  =\sum_{(y_{1},y_{2})\in\mathcal{S}^{2}}\Gamma_{(x_{2},x_{1}),(y_{2},y_{1}%
)}^{\infty}\\
&  =q^{\infty}(x_{2},x_{1}).
\end{align*}
The extension to $K>2$ is straightforward.
\end{proof}

\medskip As is discussed at the beginning of Section \ref{sec:LD}, standard
arguments using detailed balance show that the stationary distribution for
$\mathbf{Y}^{\infty}$ is $\bar{\mu}$, a symmetrized version of $\mu$:
\[
\bar{\mu}(y_{1},y_{2})=\frac{1}{2}[\mu(y_{1},y_{2})+\mu(y_{2},y_{1})].
\]
It follows that the weighted empirical measure
\begin{equation}
\eta_{T}^{\infty}=\frac{1}{T}\int_{0}^{T}\left[  \rho(Y_{1}^{\infty}%
(t),Y_{2}^{\infty}(t))\delta_{(Y_{1}^{\infty}(t),Y_{2}^{\infty}(t))}%
+\rho(Y_{2}^{\infty}(t),Y_{1}^{\infty}(t))\delta_{(Y_{2}^{\infty}%
(t),Y_{1}^{\infty}(t))}\right]  dt \label{eqn:weightedEM}%
\end{equation}
converges to $\mu$ as $T\rightarrow\infty$: using (\ref{eqn:rho}), for any
test function $f$%
\begin{align}
&  \sum_{(y_{1},y_{2})\in S^{2}}\left[  f(y_{1},y_{2})\rho(y_{1}%
,y_{2})+f(y_{2},y_{1})\rho(y_{2},y_{1})\right]  \bar{\mu}(y_{1},y_{2}%
)\label{eqn:testf}\\
&  \quad=2\sum_{(y_{1},y_{2})\in S^{2}}f(y_{1},y_{2})\rho(y_{1},y_{2})\frac
{1}{2}\left[  \mu(y_{1},y_{2})+\mu(y_{2},y_{1})\right] \nonumber\\
&  \quad=\sum_{(y_{1},y_{2})\in S^{2}}f(y_{1},y_{2})\mu(y_{1},y_{2}).\nonumber
\end{align}
Thus, $\eta_{T}^{\infty}$ can be used as an approximation of $\mu$ and for
computing ergodic averages of thermodynamic properties of the original process.

In the case of $K$ temperatures, there are $K$ rate matrices $\Gamma^{i}$. Let
$\mu$ denote the product measure of the associated $\mu_{i}$'s. The analogue
of the claim just made still holds and $\eta_{T}^{\infty}$ takes the form
\[
\eta_{T}^{\infty}=\frac{1}{T}\int_{0}^{T}\sum_{\sigma\in\Sigma_{K}}%
\rho((\mathbf{Y}^{\infty}(t))^{\sigma})\delta_{(\mathbf{Y}^{\infty
}(t))^{\sigma}}dt,
\]
where $\Sigma_{K}$ is the set of permutations of $\{1,2,\dots,K\}$ and
$\mathbf{Y}^{\infty}$ is now a Markov process that has generator
\eqref{eq:generatorINS} and with $\Gamma^{\infty}$ defined accordingly.

For notational simplicity, in general we do not distinguish between two or
more temperatures. Whenever a proof is provided only for $K = 2$, unless
otherwise stated, the reader should convince themselves that extending the
result to an arbitrary number of temperatures is a straightforward task.

\subsection{Particle-temperature associations}

As mentioned in the Introduction, when running parallel tempering or infinite
swapping algorithms, in addition to the empirical measure used for computing
ergodic averages one can record the empirical measure on the
particle-temperature associations, i.e., the fraction of time that a given
particle in the $\mathbf{Y}^{a}$ or $\mathbf{Y}^{\infty}$ formulation is
assigned the dynamic $\Gamma^{i}$.

To discuss particle-temperature associations it is useful to consider first
the pre-limit process $\mathbf{Y}^{a}$. Recall from the previous subsection
that $\mathbf{Y}^{a}$ is the process for which the dynamics, i.e., the rate
matrices, associated with each particle (component of $\mathbf{Y}^{a}$) are
changed according to swaps attempted with intensity $a$, and with the
probability of success for each attempt given by $b$ in \eqref{eq:g}. We think
of each permutation as a mapping from $\{1,2,\dots,K\}$ onto itself. One can
imagine starting with $\sigma$ equal to the identity permutation, and updating
$\sigma$ each time a swap of the dynamics occurs. Thus at each moment of time
there is a particular permutation, $\sigma$, that provides the current
assignment of rate matrices to the components $Y_{1}^{a},\dots,Y_{K}^{a}$:
$\Gamma^{\sigma^{-1}(i)}$ is the rate matrix associated with particle
$Y_{i}^{a}$.

For the pre-limit processes, the particle-temperature associations are defined
as the fraction of time that a certain permutation $\sigma\in\Sigma_{K}$ is
used to associate rate matrices to the components of the process. This is then
used to create an empirical measure, $\rho_{T}=(\rho_{T}^{\sigma_{1}}%
,\dots,\rho_{T}^{\sigma_{K!}})$ (for some ordering of the elements of
$\Sigma_{K}$) on the set $\Sigma_{K}$.

Although the possibility of recording what rate matrix is associated with what
component is blurred in the infinite swapping limit, the definition of the
particle-temperature associations still make sense: For each permutation
$\sigma\in\Sigma_{K}$, the corresponding component $\rho_{T}^{\sigma}$ of
$\rho_{T}$ is defined by
\[
\rho_{T}^{\sigma}=\frac{1}{T}\int_{0}^{T}\rho((\mathbf{Y}^{\infty}%
(t))^{\sigma})dt,\ \rho_{T}=\{\rho_{T}^{\sigma}\}_{\sigma\in\Sigma_{K}}.
\]
Thus $\rho_{T}$ is probability measure that puts mass $\rho_{T}^{\sigma}$ on
the permutation $\sigma$. In the case of two temperatures we typically denote
the components by $\rho_{T}=(\rho_{T}^{1},\rho_{T}^{2})$, where the
superscript $1$ refers to the identity map $\sigma=(1,2)$ and $2$ to the
permutation that reverses components, $\sigma=(2,1)$. Note that $\rho_{T}$ is
a vector with $|\Sigma_{K}| = K!$ entries,
each providing the fraction of time the corresponding permutation has been in
use up to time $T$. Since there must always be some permutation that is in use
this vector is a probability measure on $\Sigma_{K}$.

To further understand the role of $\rho_{T}$ consider again $\mathbf{Y}^{a}$,
the process associated with parallel tempering with swap rate $a$ (and with
swapped dynamics). It is not hard to show that, if the state of $\mathbf{Y}%
^{a}$ is frozen at some $\mathbf{y}$, then the swap mechanism introduces an
ergodic Markov chain on $\Sigma_{K}$, with $\rho(\mathbf{y}^{\sigma})$ the
stationary probability to be in permutation $\sigma$. Indeed, in the case
$K=2$, if we label the permutations $\sigma_{1}$ and $\sigma_{2}$, a
successful swap from $\sigma_{1}$ to $\sigma_{2}$ has probability $1\wedge
(\mu(y_{2},y_{1})/\mu(y_{1},y_{2}))$ and a successful swap in the reverse
direction has probability $1\wedge(\mu(y_{1},y_{2})/\mu(y_{2},y_{1}))$.
Without loss of generality we can assume that the former probability is
$\mu(y_{2},y_{1})/\mu(y_{1},y_{2})$, so that the probability of a successful
swap from $\sigma_{2}$ to $\sigma_{1}$ is 1. It is clear that these
transitions form a Markov chain on $\Sigma_{2}=\{\sigma_{1},\sigma_{2}\}$ and
under the assumption on the swap probabilities the transition matrix is
\[
\left(
\begin{array}
[c]{cc}%
1-\frac{\mu(y_{2},y_{1})}{\mu(y_{1},y_{2})} & \frac{\mu(y_{2},y_{1})}%
{\mu(y_{1},y_{2})}\\
1 & 0
\end{array}
\right)  .
\]
It is easy to check that the associated invariant measure is that which puts
probability $\rho(y_{1},y_{2})$ on permutation $\sigma_{1}$ and $\rho
(y_{2},y_{1})$ on $\sigma_{2}$. The generalization to arbitrary $K\geq2$ is
straightforward, albeit notationally cumbersome.

Thus for a fixed $\mathbf{y}$, $\rho(\mathbf{y}^{\sigma})$ can be interpreted
as the asymptotic fraction of time that temperatures are assigned according to
$\sigma$ when $\mathbf{Y}^{\infty}(t)=\mathbf{y}$. Under the invariant
distribution $\bar{\mu}$ of $\mathbf{Y}^{\infty}$ the expectation of
$\rho((\mathbf{Y}^{\infty}(t))^{\sigma})$ is
\begin{align*}
E\left[  \rho((\mathbf{Y}^{\infty}(t))^{\sigma})\right]   &  =\sum
_{\mathbf{y}\in\mathcal{S}^{K}}\rho(y_{\sigma^{-1}(1)},\dots,y_{\sigma
^{-1}(K)})\bar{\mu}(y_{1},\dots,y_{K})\\
&  =\sum_{\mathbf{y}\in\mathcal{S}^{K}}\left(  \frac{\mu(y_{\sigma^{-1}%
(1)},\dots,y_{\sigma^{-1}(K)})}{\sum_{\bar{\sigma}\in\Sigma_{K}}\mu
(y_{\bar{\sigma}^{-1}(1)},\dots,y_{\bar{\sigma}^{-1}(K)})}\right) \\
&  \qquad\quad\times\frac{1}{K!}\frac{\sum_{\bar{\sigma}\in\Sigma_{K}}%
\mu(y_{\bar{\sigma}^{-1}(1)},\dots,y_{\bar{\sigma}^{-1}(K)})}{\prod_{k=1}%
^{K}Z_{k}}\\
&  =\frac{1}{K!}\sum_{\mathbf{y}\in\mathcal{S}^{K}}\frac{\mu(y_{\sigma
^{-1}(1)},\dots,y_{\sigma^{-1}(K)})}{\prod_{k=1}^{K}Z_{k}}\\
&  =\frac{1}{K!}.
\end{align*}
That is, under the invariant distribution $\bar{\mu}$ all permutations are
assigned the same probability. It follows from the ergodic theorem that
$\rho_{T}$ converges to the uniform distribution on $\Sigma_{K}$.

\section{Large deviation properties}

\label{sec:LD}To use $\rho_{T}$ (or related quantities) to understand
$\eta_{T}^{\infty}$, we need to study the joint distribution of $\eta
_{T}^{\infty}$ and $\rho_{T}$ as $T$ grows to infinity. More precisely, we
will study the asymptotic properties of this joint distribution by means of
large deviations. In \cite{dupliupladol} the large deviation properties
associated with infinite swapping are explored by considering the measure
$\eta_{T}^{\infty}$ in the limit as $T$ goes to infinity. Here, the starting
point is instead the ordinary empirical measure associated with the infinite
swapping process $\mathbf{Y}^{\infty}$:
\[
\nu_{T}(\cdot)\doteq\frac{1}{T}\int_{0}^{T}\delta_{\mathbf{Y}^{\infty}%
(t)}(\cdot)dt.
\]

As noted in the previous section, $\mathbf{Y}^{\infty}$ is a continuous time
pure jump process with generator given in \eqref{eq:generatorINS}. It is
assumed that the rate matrices $\Gamma^{1}$, $\Gamma^{2}$ are reversible with
respect to $\mu_{1}$ and $\mu_{2}$, respectively. Under this assumption it is
straightforward to show that $\Gamma^{\infty}$ is reversible with respect to
$\bar{\mu}$. Without loss of generality, for some $\mathbf{x}=(x_{1},x_{2}%
)\in\mathcal{S}^{2}$ take $\mathbf{y}=(y_{1},x_{2})$, so that
\[
\Gamma_{\mathbf{x},\mathbf{y}}^{\infty}=\rho(x_{1},x_{2})\Gamma_{x_{1},y_{1}%
}^{1}+\rho(x_{2},x_{1})\Gamma_{x_{1},y_{1}}^{2}.
\]
It then follows that
\begin{align*}
\bar{\mu}(\mathbf{x})\Gamma_{\mathbf{x},\mathbf{y}}^{\infty}  &  =\frac{1}%
{2}\left[  \mu((x_{1},x_{2}))+\mu((x_{2},x_{1}))\right]  \left(  \rho
(x_{1},x_{2})\Gamma_{x_{1},y_{1}}^{1}+\rho(x_{2},x_{1})\Gamma_{x_{1},y_{1}%
}^{2}\right) \\
&  =\frac{1}{2}\mu((x_{1},x_{2}))\Gamma_{x_{1},y_{1}}^{1}+\frac{1}{2}%
\mu((x_{2},x_{1}))\Gamma_{x_{1},y_{1}}^{2}\\
&  =\frac{1}{2}\mu((y_{1},x_{2}))\Gamma_{y_{1},x_{1}}^{1}+\frac{1}{2}%
\mu((x_{2},y_{1}))\Gamma_{y_{1},x_{1}}^{2}\\
&  =\bar{\mu}((y_{1},x_{1}))\left(  \rho(y_{1},x_{2})\Gamma_{y_{1},x_{1}}%
^{1}+\rho(x_{2},y_{1})\Gamma_{y_{1},x_{1}}^{2}\right) \\
&  =\bar{\mu}(\mathbf{y})\Gamma_{\mathbf{y},\mathbf{x}}^{\infty},
\end{align*}
where we have used the definitions of $\rho$ and $\Gamma^{\infty}$, $\mu
(x_{1},x_{2})=\mu_{1}(x_{1})\mu_{2}(x_{2})$, and the reversibility of the rate
matrices with respect to $\mu_{1}$ and $\mu_{2}$. For other choices of
$\mathbf{y}$ the calculations are completely analogous.

In addition to reversibility, for the purpose of large deviation results we
assume that $\mu_{1}$ and $\mu_{2}$ are the unique invariant measures of the
Markov processes with rate matrices $\Gamma^{1}$ and $\Gamma^{2}$,
respectively. In the finite state setting considered here, uniqueness of the
invariant distributions can be ensured by assuming that all states
communicate. Note that under this condition $\mu(\mathbf{x})>0$ for all
$\mathbf{x}\in\mathcal{S}^{K}$. For a general compact state space
$\mathcal{S}$ see \cite{dupliu} for a set of sufficient conditions for the
large deviation results to hold.

By the results of \cite{dupliu}, the sequence $\{\nu_{T}\}_{T}$ satisfies a
large deviation principle on $\mathcal{P}(\mathcal{S}^{K})$ with convex and
lower semicontinuous rate function $J$ given by
\[
J(\nu)=\sum_{\mathbf{x}\in\mathcal{S}^{K}}q^{\infty}(\mathbf{x})\theta
(\mathbf{x})\bar{\mu}(\mathbf{x})-\sum_{\mathbf{x},\mathbf{y}\in
\mathcal{S}^{K}}\theta^{1/2}(\mathbf{x})\theta^{1/2}(\mathbf{y})\Gamma
_{\mathbf{x},\mathbf{y}}^{\infty}\bar{\mu}(\mathbf{x}),
\]
where $\theta(\mathbf{x})=\nu(\mathbf{x})/\bar{\mu}(\mathbf{x})$,
$\mathbf{x}\in\mathcal{S}^{K}$. We show in Lemma \ref{lem:strict_convex} in
the appendix that $J$ is strictly convex.

The empirical measure $\nu_{T}^{\infty}$ is of interest because although it is
$\eta_{T}^{\infty}$ that is used for all computational purposes, in an
infinite swapping algorithm it is $\nu_{T}^{\infty}$ that is simulated and
from which one obtains $\eta_{T}^{\infty}$. The explicit connection between
the two empirical measures is through the mapping $M:\mathcal{P}%
(\mathcal{S}^{K})\rightarrow\mathcal{P}(\mathcal{S}^{K})$ given by
\begin{equation}
(M\nu)(\mathbf{x})=\rho(x_{1},\dots,x_{K})\sum_{\sigma\in\Sigma_{K}}%
\nu(x_{\sigma^{-1}(1)},\dots,x_{\sigma^{-1}(K)}), \label{eq:mapM}%
\end{equation}
for $\mathbf{x}\in\mathcal{S}^{K}$, $\nu\in\mathcal{P}(\mathcal{S}^{K})$. For
$K=2$ the definition simplifies to
\begin{equation}
(M\nu)(x_{1},x_{2})=\rho(x_{1},x_{2})[\nu(x_{1},x_{2})+\nu(x_{2},x_{1})].
\label{eq:mapM2}%
\end{equation}
To see that \eqref{eq:mapM} maps probability measures to probability measures,
take any $\nu\in\mathcal{P}(\mathcal{S}^{K})$ [see (\ref{eqn:testf})]. Then
since $\sum_{\sigma\in\Sigma_{K}}\rho^{\sigma}(\mathbf{x})=1$ for any
$\mathbf{x}\in\mathcal{S}$,
\begin{align*}
\sum_{\mathbf{x}\in\mathcal{S}^{K}}(M\nu)(\mathbf{x})  &  =\sum_{\mathbf{x}%
\in\mathcal{S}^{K}}\left(  \rho(\mathbf{x})\sum_{\sigma\in\Sigma_{K}}%
\nu(x_{\sigma^{-1}(1)},\dots,x_{\sigma^{-1}(K)})\right) \\
&  =\sum_{\sigma\in\Sigma_{K}}\sum_{\mathbf{x}\in\mathcal{S}^{K}}%
\rho(\mathbf{x})\nu(\mathbf{x}^{\sigma})\\
&  =\sum_{\sigma\in\Sigma_{K}}\sum_{\mathbf{x}\in\mathcal{S}^{K}}%
\rho(\mathbf{x}^{\sigma})\nu(\mathbf{x})\\
&  =\sum_{\mathbf{x}\in\mathcal{S}^{K}}\sum_{\sigma\in\Sigma_{K}}%
\rho(\mathbf{x}^{\sigma})\nu(\mathbf{x})\\
&  =1.
\end{align*}
Thus $M\nu$ is indeed a probability measure on $\mathcal{S}^{K}$. Furthermore,
it is easy to check that $\eta_{T}^{\infty}=M\nu_{T}$ and $\mu=M\bar{\mu}$. In
general, $M$ is the map that takes a measure associated with the symmetrized
variables and maps it to a corresponding measure for the unsymmetrized variables.

The form of $M$ for the case of two temperatures, given in \eqref{eq:mapM2},
highlights a sort of symmetric property of the mapping: As long as the total
mass under $\nu$ of the two points $(x,y)$ and $(y,x)$ is kept fixed, the
image measure $M\nu$ will remain the same. The analogue of this of course
holds for $K>2$ as well and it turns out to be an essential property for
studying infinite swapping.

Once the mapping $M$ that takes $\nu_{T}$ to $\eta_{T}^{\infty}$ has been
identified, together with the large deviations principle for $\nu_{T}$, the
joint large deviation principle for $(\eta_{T}^{\infty},\rho_{T})$ is obtained
through an application of the contraction principle. For a function $f$ and
probability measure $\gamma$ on $\mathcal{S}^{K}$ let $\langle f,\gamma
\rangle=\sum_{\mathbf{x}\in\mathcal{S}^{K}}f(\mathbf{x})\gamma(\mathbf{x})$.

\begin{proposition}
\label{prop:jointLDP} Suppose that $\mu_{1}$ and $\mu_{2}$ are the unique
invariant measures for $\Gamma^{1}$ and $\Gamma^{2}$, respectively. The
sequence $\{(\eta_{T}^{\infty},\rho_{T})\}$ satisfies a large deviation
principle on $\mathcal{P}(\mathcal{S}^{K})\times\mathcal{P}(\Sigma_{K})$ with
rate function
\[
I(\gamma,w)=\inf\left\{  J(\nu):\ \gamma=M\nu,\ \langle\rho^{\sigma}%
,\nu\rangle=w_{\sigma},\ \sigma\in\Sigma_{K}\right\}  .
\]

\end{proposition}

\begin{proof}
The map $M$ is continuous with respect to the weak topology (which is the same
as the standard Euclidean topology if we think of $\mathcal{P}(\mathcal{S}%
^{K})$ as embedded into $\mathbb{R}^{\left\vert \mathcal{S}\right\vert ^{K}}%
$). Furthermore, the components of $\rho_{T}$ can be expressed as expectations
with respect to $\nu_{T}$:
\[
\rho_{T}^{\sigma}=\sum_{\mathbf{x}\in\mathcal{S}^{K}}\rho(\mathbf{x}^{\sigma
})\nu_{T}(\mathbf{x})=\langle\rho^{\sigma},\nu_{T}\rangle,\ \sigma\in
\Sigma_{K}.
\]
Consider the map from $\mathcal{P}(\mathcal{S}^{K})$ to $\mathcal{P}%
(\Sigma_{K})$ defined by $\nu\mapsto(\langle\rho^{\sigma_{1}},\nu\rangle
,\dots,\langle\rho^{\sigma_{K!}},\nu\rangle)$. This is a continuous map with
respect to the weak topology and precisely the one that takes $\nu_{T}$ to
$\rho_{T}$. Thus, the pair $(\eta_{T}^{\infty},\rho_{T})$ is obtained by
applying a continuous map to $\nu_{T}$. It follows from the contraction
principle that the sequence satisfies a large deviation principle with the
prescribed rate function as $T\rightarrow\infty$ .
\end{proof}

\medskip In this paper we start with the symmetrized process $\mathbf{Y}%
^{\infty}$ and its associated empirical measure, and then study various large
deviation properties of infinite swapping through mappings and the contraction
principle. In \cite{dupliupladol} the large deviation principle for $\eta
_{T}^{\infty}$ is proved without reference to the symmetrized process (though
for diffusion processes rather than the jump processes discussed here). For
completeness some comments on the relation between the rate directly obtained
in \cite{dupliupladol} (more precisely the analogous rate appropriate for the
model considered here) and the joint large deviation principle of Proposition
\ref{prop:jointLDP} are appropriate.

Note that from the joint large deviation principle one immediately obtains
large deviation principles for the marginals $\eta_{T}^{\infty}$ and $\rho
_{T}$. Denote the corresponding rate functions $I_{1}$ and $I_{2}$:
\[
I_{1}(\gamma)=\inf\left\{  J(\nu):\ M\nu=\gamma\right\}  ,\ \gamma
\in\mathcal{P}(\mathcal{S}^{K}),
\]
and
\begin{equation}
\label{eq:I2}I_{2}(w)=\inf\left\{  J(\nu):\ \langle\rho^{\sigma},\nu
\rangle=w_{\sigma},\ \sigma\in\Sigma_{K}\right\}  ,\ w=\{w_{\sigma}%
\}\in\mathcal{P}(\Sigma_{K}).
\end{equation}
Let $I^{\infty}$ denote the rate function analogous to that of
\cite{dupliupladol} that would be appropriate for $\{\eta_{T}^{\infty}\}$,
\begin{align*}
I^{\infty} (\gamma) = \sum_{\mathbf{x}\in\mathcal{S}^{2}} q(\mathbf{x})
\gamma(\mathbf{x}) - \sum_{\mathbf{x}, \mathbf{y} \in\mathcal{S}^{2}} \left[
\frac{d\gamma}{d\mu}\right]  ^{1/2}(\mathbf{x}) \left[  \frac{d\gamma}{d\mu
}\right]  ^{1/2}(\mathbf{y}) \Gamma_{\mathbf{x}, \mathbf{y}} \mu(\mathbf{x})
\end{align*}
where $\Gamma$ is the rate matrix associated with the original uncoupled
processes. This rate function is finite only for those measures $\gamma$ that
satisfy, for $K=2$,
\begin{equation}
\lbrack d\gamma/d\mu](x_{1},x_{2})=[d\gamma/d\mu](x_{2},x_{1}).
\label{eqn:sym}%
\end{equation}
This constraint is immediately satisfied by any measure $\gamma$ on the form
$\gamma=M\nu$, for some $\nu\in\mathcal{P}(\mathcal{S}^{2})$, and thus it is
built in to the definitions of $I$ and $I_{1}$. That is, if $\gamma$ violates
this condition of correct relative weights on $(x_{1},x_{2})$ and
$(x_{2},x_{1})$, then there is no $\nu$ for which $\gamma=M\nu$ and both $I$
and $I_{1}$ are by definition infinite for such measures. To see that this is
true is a short calculation using only the definition of $M$: Take $\gamma
\in\mathcal{P}(\mathcal{S}^{2})$ for which there is some $\nu$ such that
$\gamma=M\nu$ and $\gamma\ll\mu$. Since these are discrete measures, for any
$(x_{1},x_{2})\in\mathcal{S}^{2}$,
\begin{align*}
\left[  \frac{d\gamma}{d\mu}\right]  (x_{1},x_{2})  &  =\frac{\gamma
(x_{1},x_{2})}{\mu(x_{1},x_{2})}\\
&  =\frac{(M\nu)(x_{1},x_{2})}{\mu(x_{1},x_{2})}\\
&  =\frac{\rho(x_{1},x_{2})[\nu(x_{1},x_{2})+\nu(x_{2},x_{1})]}{\mu
(x_{1},x_{2})}\\
&  =\frac{\nu(x_{1},x_{2})+\nu(x_{2},x_{1})}{2\bar{\mu}(x_{1},x_{2})},
\end{align*}
and it follows by symmetry that $[d\gamma/d\mu](x_{2},x_{1})$ is equal to this
as well. It should be clear that the relation also holds for $K>2$.
Furthermore, it continues to be true in the case of an uncountable state space
$\mathcal{S}$, most easily seen by assuming that all measures involved have
densities with respect to some common reference measure and using the same
kind of argument as here.

Although the INS dynamics use symmetrized dynamics, there is no reason that
its empirical measure must be symmetric. However, if $\gamma$ is a point in
the support of the weighted empirical measure (\ref{eqn:weightedEM}), then the
most likely empirical measure for the INS process that leads to $\gamma$ will
be symmetric. Thus for probability measures $\gamma$ that satisfy the weighted
symmetry condition (\ref{eqn:sym}) we have the following result.

\begin{proposition}
\label{prop:symMeasure} Suppose $\gamma\in\mathcal{P}(\mathcal{S}^{K})$ is
such that $\mathcal{M}(\gamma)=\{\nu\in\mathcal{P}(\mathcal{S}^{K}%
):\gamma=M\nu\}$ is non-empty. Then
\[
\inf\{J(\nu):\nu\in\mathcal{M}(\gamma)\}
\]
is attained at the symmetric $\nu_{sym}$ given by
\begin{equation}
\nu_{sym}(x_{1},\dots,x_{K})=\frac{\gamma(x_{1},\dots,x_{K})}{K!\rho
(\mathbf{x})}. \label{eq:symNu}%
\end{equation}

\end{proposition}

Before proceeding with the proof a short remark on the form of $\nu_{sym}$ is
be in place. Although not necessarily apparent at first, $\nu_{sym}$ is indeed
symmetric:
\[
\nu_{sym}(x_{1},\dots,x_{K})=\nu_{sym}(x_{\sigma^{-1}(1)},\dots,x_{\sigma
^{-1}(K)}),
\]
for any $\mathbf{x}\in\mathcal{S}^{K}$ and permutation $\sigma\in\Sigma_{K}$.
This follows from the weighted symmetry of $\gamma$ as expressed in
(\ref{eqn:sym}) which holds for any measure obtained as a mapping through $M$,
and the definition of $\rho$. For the sake of clarity we show for $K=2$ that
under the mapping $M$, $\nu_{sym}$ returns the measure $\gamma$. Indeed, for
any $(x_{1},x_{2})\in\mathcal{S}^{2}$,
\begin{align*}
\left(  M\nu_{sym}\right)  (x_{1},x_{2})  &  =\rho(x_{1},x_{2})\left[
\nu_{sym}(x_{1},x_{2})+\nu_{sym}(x_{2},x_{1})\right] \\
&  =\rho(x_{1},x_{2})\left(  \frac{\gamma(x_{1},x_{2})}{2\rho(x_{1},x_{2}%
)}+\frac{\gamma(x_{2},x_{1})}{2\rho(x_{2},x_{1})}\right) \\
&  =\frac{\mu(x_{1},x_{2})}{2}\left(  \frac{\gamma(x_{1},x_{2})}{\mu
(x_{1},x_{2})}+\frac{\gamma(x_{2},x_{1})}{\mu(x_{2},x_{1})}\right)  .
\end{align*}
The condition (\ref{eqn:sym}) on the relative weights $[d\gamma/d\mu]$ then
shows that this is equal to $\gamma(x_{1},x_{2})$.

Since the rate function of a large deviation principle is unique the following
result is to be expected.

\begin{corollary}
\label{cor:rateFcn} For any $\gamma\in\mathcal{P}(\mathcal{S}^{K})$,
$I_{1}(\gamma) = I^{\infty}(\gamma)$.
\end{corollary}

The result follows almost immediately from Proposition \ref{prop:symMeasure}
by inserting the symmetric measure $\nu_{sym}$ into $J$, using the definitions
of $\gamma^{\infty}$ and $\Gamma^{\infty}$ and the constraint \eqref{eqn:sym}.
The calculations are straightforward but cumbersome and are left out for
brevity.\medskip

\begin{proof}
[Proof of Proposition \ref{prop:symMeasure}]Due to the form of $M$,
$\mathcal{M}(\gamma)$ is a closed, convex set in $\mathcal{P}(\mathcal{S}%
^{K})$. Thus, if $\mathcal{M}(\gamma)$ is non-empty, then by strict convexity
$J$ will uniquely attain its infimum over the set. Using this fact we argue by
convexity that the minimizing measure must indeed be the symmetric $\nu_{sym}$.

For any $\nu\in\mathcal{M}(\gamma)$ let $\nu^{\sigma}$ be a permutation of
$\nu$ according to $\sigma$. For any set $A\subset\mathcal{S}^{K}$,
\[
\nu^{\sigma}(A)=\sum_{\mathbf{x}\in\mathcal{S}^{K}}I\{(x_{1},\dots,x_{K})\in
A\}\nu(d\mathbf{x}^{\sigma}),
\]
or, equivalently, $\nu^{\sigma}(A)=\nu(A^{\sigma})$, where $A^{\sigma
}=\{\mathbf{x}^{\sigma}:\mathbf{x}\in A\}$. The symmetry property of $M$,
discussed after \eqref{eq:mapM2}, ensures that $\nu^{\sigma}\in\mathcal{M}%
(\gamma)$ for every $\sigma\in\Sigma_{K}$. Moreover, since the measure
$\bar{\mu}$ is symmetric, $\nu\ll\bar{\mu}$ implies $\nu^{\sigma}\ll\bar{\mu}%
$. The key to the proof is to show that $J(\nu)=J(\nu^{\sigma})$. Indeed,
assume that this holds and that $\bar{\nu}$ is the unique minimizing measure
in $\mathcal{M}(\gamma)$,
\[
\bar{\nu}=\underset{\nu\in\mathcal{M}(\gamma)}{\operatorname{argmin}}%
\ J(\nu).
\]
If $\bar{\nu}$ is not symmetric, then there is a $\sigma\in\Sigma_{K}$ for
which $\bar{\nu}\neq\bar{\nu}^{\sigma}$ and it follows that $\bar{\nu}%
\neq(1/2)\bar{\nu}+(1/2)\bar{\nu}^{\sigma}\in\mathcal{M}(\gamma)$. By strict
convexity of $J$ and the assumption that $J(\bar{\nu})=J(\bar{\nu}^{\sigma}%
)$,
\[
J\left(  (1/2)\bar{\nu}+(1/2)\bar{\nu}^{\sigma}\right)  <\frac{1}{2}J(\bar
{\nu})+\frac{1}{2}J(\bar{\nu}^{\sigma})=J(\nu),
\]
which contradicts $\bar{\nu}$ being the unique minimizer of $J$ over
$\mathcal{M}(\gamma)$. Hence, the minimizing $\nu$ must be symmetric and
therefore satisfy \eqref{eq:symNu}.

It remains to show that for any $\nu\in\mathcal{M}(\gamma)$, and any
permutation $\sigma$, $J(\nu)=J(\nu^{\sigma})$. Let $\eta^{\sigma}%
(\mathbf{x})=\nu^{\sigma}(\mathbf{x})/\bar{\mu}(\mathbf{x})$. Then
$J(\nu)=J(\nu^{\sigma})$ follows directly from Lemma \ref{lemma:symmetry}
(symmetry properties for $q^{\infty}$ and $\Gamma^{\infty}$) together with the
definition of $\nu^{\sigma}$ and the symmetry of $\bar{\mu}$:
\begin{align*}
J(\nu^{\sigma})  &  =\sum_{\mathbf{x}\in\mathcal{S}^{K}}q^{\infty}%
(\mathbf{x})\eta^{\sigma}(\mathbf{x})\bar{\mu}(\mathbf{x})-\sum_{\mathbf{x}%
,\mathbf{y}\in\mathcal{S}^{K}}(\eta^{\sigma})^{1/2}(\mathbf{x})(\eta^{\sigma
})^{1/2}(\mathbf{y})\Gamma_{\mathbf{x},\mathbf{y}}^{\infty}\bar{\mu
}(\mathbf{x})\\
&  =\sum_{\mathbf{x}\in\mathcal{S}^{K}}q^{\infty}(\mathbf{x}^{\sigma}%
)\eta(\mathbf{x}^{\sigma})\bar{\mu}(\mathbf{x}^{\sigma})-\sum_{\mathbf{x}%
,\mathbf{y}\in\mathcal{S}^{K}}\eta^{1/2}(\mathbf{x}^{\sigma})\eta
^{1/2}(\mathbf{y}^{\sigma})\Gamma_{\mathbf{x}^{\sigma},\mathbf{y}^{\sigma}%
}^{\infty}\bar{\mu}(\mathbf{x}^{\sigma})\\
&  =J(\nu).
\end{align*}
This completes the proof.
\end{proof}

\begin{remark}
\label{remark:generalLDP} Up to this point no results have relied heavily upon
the assumption of a finite state space $S$. Indeed, the large deviation
principle of Proposition \ref{prop:jointLDP} holds in greater generality than
what is presented here. For example, one can consider a compact (this
condition can also be weakened) Polish space $S$ and let jump intensities on
this space to describe the dynamics of the processes. Then, in the case $K=2$,
as long as the measures $\mu_{1}$ and $\mu_{2}$ of interest are the unique
invariant measures for the chosen jump intensities, the large deviation
principle will hold. Sufficient conditions for this can be found in
\cite{dupliu}. Moreover, under reversibility assumptions the symmetry
properties proved in Lemma \ref{lemma:symmetry} also hold for the associated
transition kernels and thus Proposition \ref{prop:symMeasure} remains valid as well.
\end{remark}

\section{An ergodic control problem arising from infinite swapping}

\label{sec:controlProblem} The aim of this section is to introduce a finite
time stochastic control problem, as well as the corresponding ergodic control
problem, that are related to the infinite swapping process $\mathbf{Y}%
^{\infty}$. The control problem introduced here will be used in the following
sections to show asymptotic results regarding the particle-temperature
associations and discuss the behavior and performance of infinite swapping
when the underlying potential landscape exhibits asymmetry. Although all
results stated here are standard in stochastic control, we include them since
they may not be familiar to readers with experience in Monte Carlo methods.

With $\Gamma^{\infty}$ the rate matrix of the infinite swapping process
$\mathbf{Y}^{\infty}$, we simplify notation slightly by letting
$r(\boldsymbol{y},\boldsymbol{z})=\Gamma_{(\boldsymbol{y},\boldsymbol{z}%
)}^{\infty}$ for $\boldsymbol{y},\boldsymbol{z}\in\mathcal{S}^{2}$. The
generator of the process is then given by
\[
\mathcal{L}f(\boldsymbol{y})=\sum_{\boldsymbol{z}\in\mathcal{S}^{2}%
}[f(\boldsymbol{z})-f(\boldsymbol{y})]r(\boldsymbol{y},\boldsymbol{z}%
),\ \boldsymbol{y}\in\mathcal{S}^{2}.
\]
The infinite swapping process $\mathbf{Y}^{\infty}$ takes values in
$\mathcal{D}([0,\infty):\mathcal{S}^{2})$, but for each fixed $T<\infty$ we
can also consider it as an element of $\mathcal{D}([0,T]:\mathcal{S}^{2})$.
Our interest is now in evaluating the normalized expectation $\frac{1}{T}\log
E[e^{-TF(\mathbf{Y}^{\infty}(\cdot))}]$, in the limit as $T\rightarrow\infty$,
for functionals $F:\mathcal{D}([0,T]:\mathcal{S}^{2})\rightarrow\mathbb{R}$
that are of the form
\begin{equation}
F(\mathbf{Y}^{\infty})=\frac{1}{T}\int_{0}^{T}h(\mathbf{Y}^{\infty}(s))ds,
\label{eq:fcnF}%
\end{equation}
for some function $h:\mathcal{S}^{2}\rightarrow\mathbb{R}$. The reason for
studying the quantity $\frac{1}{T}\log E[e^{-TF(\mathbf{Y}^{\infty}(\cdot))}]$
is that, using the large deviation results of Section \ref{sec:LD}, the limit
($T\rightarrow\infty$) can be related to certain optimization problems, which
in turn are of interest for evaluating the performance of infinite swapping.
This is carried out in Sections \ref{sec:PT} and \ref{sec:asymmetry}.
Throughout the section the infinite swapping process is assumed to start in
some state $\mathbf{y}_{0}\in\mathcal{S}^{2}$.

\subsection{A stochastic control problem}

Take $T<\infty$ to be fixed. Because the state space is finite, in formulating
a stochastic control representation for $E[e^{-TF(\mathbf{Y}^{\infty}(\cdot
))}]$ we will be able to restrict to feedback controls. The control space will
be a collection of rates, and therefore takes the form $U\doteq\lbrack
0,\infty)^{\left\vert \mathcal{S}\right\vert ^{2}}$. Let $\mathcal{U}^{T}$ be
the space of functions $u:[0,T]\times\mathcal{S}^{2}\rightarrow U$ that are
continuous in $t$ and for which the $\mathbf{z}$th component of the vector
$u(t,\mathbf{y})$ is positive only if $r(\mathbf{y},\mathbf{z})>0$:
\begin{align*}
\mathcal{U}^{T}  &  =\left\{  u:[0,T]\times\mathcal{S}^{2}\rightarrow
U:\ u(\mathbf{y},t)\text{ continuous in }t,\right. \\
&  \qquad\left.  u(t,\mathbf{y};\boldsymbol{\mathbf{z}})>0\text{ only if
}r(\mathbf{y},\boldsymbol{\mathbf{z}})>0,t\in\lbrack0,T]\right\}  ,
\end{align*}
where $u(t,\mathbf{y};$$\boldsymbol{\mathbf{z}}$$)$ denotes the component of
the vector $u(\mathbf{y},t)\in U$ corresponding to $\mathbf{z}\in
\mathcal{S}^{2}$. Then to each control $u\in\mathcal{U}^{T}$ we will associate
a controlled process $\bar{\mathbf{Y}}^{\infty}$, where for each $t\in
\lbrack0,T]$ and $\mathbf{y}\in\mathcal{S}^{2}$, the set of jump intensities
for $\bar{\mathbf{Y}}^{\infty}$ when in state $\mathbf{y}$ at time $t$ is
given by $u(t,\mathbf{y})\in U$; the jump intensity from $\mathbf{y}$ to
$\boldsymbol{\mathbf{z}}\in\mathcal{S}^{2}$ is $u(t,\mathbf{y};$%
$\boldsymbol{\mathbf{z}}$$)$. Although $\bar{\mathbf{Y}}^{\infty}$ depends on
$u$, this is not made explicit in the notation, though the overbar indicates
we consider a controlled process rather than the original infinite swapping
process $\mathbf{Y}^{\infty}$.

We make a slight abuse of notation and also denote by $u$ the elements in $U$
that the control processes can take on. With this notation the generator of
the controlled process $\bar{\mathbf{Y}}^{\infty}$ is given by $(\mathcal{L}%
^{u(t,\mathbf{y})}f)(\mathbf{y},t)$, where
\[
(\mathcal{L}^{u}f)(\mathbf{y},t)=\sum_{\boldsymbol{\mathbf{z}}\in
\mathcal{S}^{2}}[f(\boldsymbol{\mathbf{z}})-f(\mathbf{y}%
)]u(\boldsymbol{\mathbf{z}}).
\]

To discuss existence and uniqueness (in law) of controlled processes we use a
martingale problem characterization. For a specific feedback control
$u\in\mathcal{U}^{T}$ we say that $u$ has an associated controlled process
starting at $\mathbf{y}_{0}$ at time $t$ if the following holds. On some
probability space $(\Omega,\mathcal{F},P)$, equipped with a filtration
$\{\mathcal{F}_{t}\}_{t\in\lbrack0,T]}$, there exists a Markov process
$\{\bar{\mathbf{Y}}^{\infty}(t):t\in\lbrack0,T]\}$, satisfying $\bar
{\mathbf{Y}}^{\infty}(0)=\mathbf{y}_{0}$ and, for $t>0$,
\[
f(t,\bar{\mathbf{Y}}^{\infty}(t))-f(0,\mathbf{y}_{0})-\int_{0}^{t}\left(
(\mathcal{L}^{u(r,\bar{\mathbf{Y}}^{\infty}(r))}f)(r,\bar{\mathbf{Y}}^{\infty
}(r))+f_{t}(r,\bar{\mathbf{Y}}^{\infty}(r))\right)  dr
\]
is an $\mathcal{F}_{t}$-martingale for all $f:[0,T]\times\mathcal{S}%
^{2}\rightarrow\mathbb{R}$ that are bounded and continuously differentiable in
$t$. Since $\mathcal{S}^{2}$ is a finite space and we consider controls that
are bounded in the time variable the existence of a solution to the martingale
problem is guaranteed. Indeed, for a control $u\in\mathcal{U}^{T}$, one can
explicitly construct an associated controlled process $\bar{\mathbf{Y}%
}^{\infty}$ that solves the martingale problem by taking the correct
exponential clocks etc., see \cite[Chapter 4]{ethkur}. This process is in fact
simply the jump Markov process with the given (smooth in $t$) jump rates. The
constructed process is a solution to the associated martingale problem, and
when combined with the Feller property we have that this is indeed the unique
solution \cite{rogwil,ethkur}.

A key ingredient in what will follow in this and the next two sections is the
following stochastic control representation. For $F:\mathcal{D}%
([0,T]:\mathcal{S}^{2})\rightarrow\mathbb{R}$ of the form \eqref{eq:fcnF},
\begin{align}
&  -\log E\left[  e^{-TF(\mathbf{Y}^{\infty})}\right]  \label{eq:controlRep}\\
&  \quad=\inf_{u\in\mathcal{U}^{T}}E\left[  \int_{0}^{T}\left(  \sum
_{\boldsymbol{\mathbf{z}}\in\mathcal{S}^{2}:r(\boldsymbol{z},\bar{\mathbf{Y}%
}^{\infty}(s))>0}r(\bar{\mathbf{Y}}^{\infty}(s),\boldsymbol{\mathbf{z}}%
)\ell\left(  \frac{u(s,\bar{\mathbf{Y}}^{\infty}(s);\boldsymbol{\mathbf{z}}%
)}{r(\bar{\mathbf{Y}}^{\infty}(s),\boldsymbol{\mathbf{z}})}\right)
+h(\bar{\mathbf{Y}}^{\infty}(s))\right)  ds\right]  ,\nonumber
\end{align}
where $\bar{\mathbf{Y}}^{\infty}(0)=\mathbf{y}_{0}$, the infimum is over all
controls $u$, and $\ell$ is the function
\[
\ell(x)=%
\begin{cases}
x\log x-x+1, & x\geq0,\\
\infty, & \text{otherwise}.
\end{cases}
\]
The right-hand side of \eqref{eq:controlRep} is a stochastic control problem
with running cost $c:\mathcal{S}^{2}\times U\rightarrow\mathbb{R}$ given by
\begin{equation}
c(\mathbf{y},u)=\sum_{\mathbf{z}\in\mathcal{S}^{2}:r(\mathbf{y}%
,\boldsymbol{\mathbf{z}})>0}r(\mathbf{y},\mathbf{z})\ell\left(  \frac
{u(\mathbf{z})}{r(\mathbf{y},\mathbf{z})}\right)  +h(\mathbf{y}%
).\label{eq:costFcn}%
\end{equation}
Representations such as \eqref{eq:controlRep} are commonly used in connection
with large deviations and similar results can be found in, e.g.,
\cite{boudup,buddupmar2,dupwell4,fle}.

To discuss the dynamic programming equation associated with the stochastic
control problem in \eqref{eq:controlRep}, define $\bar{W}^{T}(t,\mathbf{y})$
to be the conditional version of the right-hand side of the representation:
\begin{equation}
\bar{W}^{T}(t,\mathbf{y})\doteq\inf_{u\in\mathcal{U}^{T}}E_{t,\mathbf{y}%
}\left[  \int_{t}^{T}c\left(  \bar{\mathbf{Y}}^{\infty}(s),u(s,\bar
{\mathbf{Y}}^{\infty}(s))\right)  \right]  , \label{eqn:control_prom}%
\end{equation}
where $E_{t,\mathbf{y}}$ denotes conditional expectation with respect to
$\bar{\mathbf{Y}}^{\infty}(t)=\mathbf{y}$. Note that the representation
\eqref{eq:controlRep} then involves the conditional expectation
$E_{0,\mathbf{y}_{0}}$. Proposition \ref{prop:Verification}, which follows
from a standard verification argument, shows that any $C^{1}$ solution of the
dynamic programming equation
\begin{equation}
W_{t}^{T}(t,\mathbf{y})+\inf_{u\in U}\left\{  \mathcal{L}^{u}W^{T}%
(t,\mathbf{y})+c(\mathbf{y},u)\right\}  =0, \label{eq:DPE}%
\end{equation}
with the associated terminal condition $W^{T}(T,\mathbf{y})=0$ for
$\mathbf{y}\in\mathcal{S}^{2}$, is equal to $\bar{W}^{T}$.

Before stating and proving the verification theorem just eluded to, we
consider the optimal control in \eqref{eq:DPE}. From the definition of
$\mathcal{L}^{u}$ it follows that, for any $u\in U$,
\[
\mathcal{L}^{u}W^{T}(t,\mathbf{y})=\sum_{\mathbf{z}\in\mathcal{S}^{2}}\left[
W^{T}(t,\mathbf{z})-W^{T}(t,\mathbf{y})\right]  u(\mathbf{z}).
\]
Straightforward calculus shows that the minimizing vector $u$ in
\eqref{eq:DPE} is given by
\begin{equation}
\bar{u}(t,\mathbf{y};\boldsymbol{\mathbf{z}})=r(\mathbf{y}%
,\boldsymbol{\mathbf{z}})e^{-(W^{T}(t,\boldsymbol{\mathbf{z}})-W^{T}%
(t,\mathbf{y}))},\ \boldsymbol{\mathbf{z}}\in\mathcal{S}^{2}.
\label{eqn:opt_u}%
\end{equation}
Inserting this expression into (\ref{eq:DPE}) gives a total of $N^{2}$ coupled
differential equations for $W^{T}$, one for each $\mathbf{y}\in\mathcal{S}%
^{2}$, which we write as
\begin{equation}
W_{t}^{T}(t,\mathbf{y})=-\sum_{z\in\mathcal{S}^{2}}r(\mathbf{y}%
,\boldsymbol{\mathbf{z}})\left(  1-e^{-(W^{T}(t,\boldsymbol{\mathbf{z}}%
)-W^{T}(t,\mathbf{y}))}\right)  -h(\mathbf{y}). \label{eqn:ODEs}%
\end{equation}
We are now ready to state the relevant verification result. The argument is
standard, but is included for completeness.

\begin{proposition}
\label{prop:Verification} Suppose that $W^{T}:[0,T]\times\mathcal{S}%
^{2}\rightarrow\mathbb{R}$ is a $C^{1}$ (in $t$) solution to the dynamic
programming equation \eqref{eq:DPE}. Define $\bar{u}(t,\mathbf{y}%
;\boldsymbol{\mathbf{z}})\in\mathcal{U}^{T}$ by (\ref{eqn:opt_u}), and let
$\bar{\mathbf{Y}}^{\infty}$ be the corresponding jump Markov process (i.e.,
solution to the martingale problem). Then $W^{T}$ equals the minimal cost
function $\bar{W}^{T}$ defined in (\ref{eqn:control_prom}), and
\[
W^{T}(t,\mathbf{y})=E_{t,\mathbf{y}}\left[  \int_{t}^{T}c\left(
\bar{\mathbf{Y}}^{\infty}(s),\bar{u}(s,\bar{\mathbf{Y}}^{\infty}(s))\right)
ds\right]  ,\ (t,\mathbf{y})\in\lbrack0,T]\times\mathcal{S}^{2},
\]
so that $\bar{u}\in\mathcal{U}^{T}$ is the optimal control and $\bar
{\mathbf{Y}}^{\infty}$ is the optimally controlled process.
\end{proposition}

\begin{proof}
To emphasize the choice of control, let $E^{u}$ denote expectation when the
control $u\in\mathcal{U}^{T}$ is used and take $\tilde{u}\in\mathcal{U}^{T}$
to be any control with associated controlled process $\tilde{\mathbf{Y}%
}^{\infty}$. From the martingale property it follows that
\begin{align*}
&  E_{t,\mathbf{y}}^{\tilde{u}}\left[  W^{T}(T,\tilde{\mathbf{Y}}^{\infty
}(T))\right] \\
&  \quad=W^{T}(t,\mathbf{y})+E_{t,\mathbf{y}}^{\tilde{u}}\left[  \int_{t}%
^{T}\left(  \mathcal{L}^{\tilde{u}(s,\tilde{\mathbf{Y}}^{\infty}(s))}%
W^{T}(s,\tilde{\mathbf{Y}}^{\infty}(s))+W_{t}^{T}(s,\tilde{\mathbf{Y}}%
^{\infty}(s))\right)  ds\right]  .
\end{align*}
The terminal condition is $W^{T}(T,\mathbf{x})=0$ for all $\mathbf{x}%
\in\mathcal{S}^{2}$ and the left-hand side is thus $0$. Moreover, since
$W^{T}$ solves the dynamic programming equation \eqref{eq:DPE} we have%
\[
W_{t}^{T}(t,\mathbf{y})+\mathcal{L}^{\tilde{u}(t,\mathbf{y})}W^{T}%
(t,\mathbf{y})\geq-c(\mathbf{y},\tilde{u}(t,\mathbf{y})),
\]
and therefore obtain a lower bound on the cost:
\[
W^{T}(t,\mathbf{y})\leq E_{t,\mathbf{y}}^{\tilde{u}}\left[  \int_{t}%
^{T}c(\tilde{\mathbf{Y}}^{\infty}(s),\tilde{u}(s,\tilde{\mathbf{Y}}^{\infty
}(s)))ds\right]  .
\]
If we instead use the optimal control $\bar{u}$, for which the infimum in
\eqref{eq:DPE} is attained, all inequalities become equalities, and therefore
\[
W^{T}(t,\mathbf{y})=E_{t,\mathbf{y}}^{\bar{u}}\left[  \int_{t}^{T}%
c(\bar{\mathbf{Y}}^{\infty}(s),\bar{u}(s,\bar{\mathbf{Y}}^{\infty
}(s)))ds\right]  .
\]
This shows that $W^{T}$ is indeed the minimal cost function and $\bar{u}$ is
the optimal control.
\end{proof}

\medskip Before taking the limit as $T\rightarrow\infty$, suppose we have a
solution $W^{T}$ to \eqref{eq:DPE} and consider the function $V:[0,T]\times
\mathcal{S}^{2}\rightarrow\mathbb{R}$ defined by $V(t,\mathbf{y})\doteq
e^{-W^{T}(t,\mathbf{y})}$. Then $W^{T}$ satisfying (\ref{eqn:ODEs}) implies
that $V$ satisfies the system
\[
0=V_{t}(t,\mathbf{y})+\sum_{\mathbf{z}\in\mathcal{S}^{2}}r(\mathbf{y}%
,\mathbf{z})\left(  V(t,\mathbf{z})-V(t,\mathbf{y})\right)  -h(\mathbf{y}%
)V(t,\mathbf{y}),
\]
with terminal condition $V(T,\mathbf{y})=1$. This is a finite system of linear
ordinary differential equations and existence and uniqueness of a solution
$V$, and thus $W^{T}$, hold. Moreover, it can be shown that in fact
$V(t,\mathbf{y})=E_{t,\mathbf{y}}[\exp-\int_{t}^{T}h(\mathbf{Y}^{\infty
}(s))ds]$ and, since $V(0,0)$ and $W^{T}(0,0)$ correspond to the two
quantities in \eqref{eq:controlRep}, uniqueness of the solution of the system
of ODEs together with Proposition \ref{prop:Verification} then confirm the
representation \eqref{eq:controlRep}.

\subsection{Limit control problem as $T \to\infty$}

As mentioned at the beginning of this section, we are ultimately interested in
the limit as $T\rightarrow\infty$ of the normalized expectation $-\frac{1}%
{T}\log E[e^{-TF(\mathbf{Y}^{\infty})}]$. Given the representation
\eqref{eq:controlRep} this is equivalent to taking the limit of
\[
\inf_{u\in\mathcal{U}^{T}}E\left[  \frac{1}{T}\int_{0}^{T}\left(  \sum
_{z\in\mathcal{S}^{2}}r(\bar{\mathbf{Y}}^{\infty}(s),z)\ell\left(
\frac{u(s,\bar{\mathbf{Y}}^{\infty}(s);z)}{r(\bar{\mathbf{Y}}^{\infty}%
(s),z)}\right)  +h(\bar{\mathbf{Y}}^{\infty}(s))\right)  ds\right]  ,
\]
where $\bar{\mathbf{Y}}^{\infty}$ is the controlled process associated with
control $u$ and the infimum is over all such feedback controls. In light of
the previous subsection, this is precisely $\lim_{T\rightarrow\infty}%
W^{T}(0,\mathbf{y}_{0})/T$, for $\mathbf{y}_{0}\in\mathcal{S}^{2}$, which
falls under the umbrella of ergodic control problems, or \textquotedblleft
average cost per unit time\textquotedblright; some general references are
\cite{fleson,kusdup1}. Since the set of initial conditions is finite
convergence will be uniform with respect to this parameter, and hence it is
not made explicit in the notation.

The limit Bellman equation is
\begin{equation}
\inf_{u\in U}\left\{  \mathcal{L}^{u}W(y)-\gamma+c(\mathbf{y},u)\right\}  =0,
\label{eq:HJeq}%
\end{equation}
where $W$ and $\gamma$ are unknown, with $\gamma$ the sought-after limit, and
$c$ is defined in \eqref{eq:costFcn}. As is well known the solution $W$ to
such an equation is unique only up to an additive constant. Together with the
form of the generator $\mathcal{L}^{u}$, the definition of $c$ implies that,
for each $\mathbf{y}\in\mathcal{S}^{2}$, the Bellman equation \eqref{eq:HJeq}
takes the form
\begin{equation}
\inf_{u\in U}\left\{  \sum_{\mathbf{z}\in\mathcal{S}^{2}}\left(
u(\mathbf{z})[W(\mathbf{z})-W(\mathbf{y})]+r(\mathbf{y},\mathbf{z})\ell\left(
\frac{u(\mathbf{z})}{r(\mathbf{y},\mathbf{z})}\right)  \right)  -\gamma
+h(\mathbf{y})\right\}  =0. \label{eq:HJeq2}%
\end{equation}
The minimizing $\bar{u}$ takes the same form as for the pre-limit problem,
\begin{equation}
\bar{u}(\mathbf{z})=r(\mathbf{y},\mathbf{z})e^{-(W(\mathbf{z})-W(\mathbf{y}%
))}, \label{eqn:opt_c_stat}%
\end{equation}
with the corresponding generator
\[
\mathcal{L}^{\bar{u}(\mathbf{z})}f(y)=\sum_{\mathbf{z}\in\mathcal{S}^{2}%
}r(\mathbf{y},\mathbf{z})e^{-(W(\mathbf{z})-W(\mathbf{y}))}[f(\mathbf{z}%
)-f(\mathbf{y})].
\]
Similar to the equation for the pre-limit control problem in the previous
subsection, inserting the optimal $\bar{u}$ in \eqref{eq:HJeq} yields the
equation
\begin{equation}
0=\sum_{z\in\mathcal{S}^{2}}r(\mathbf{y},\mathbf{z})[1-e^{-(W(\mathbf{z}%
)-W(\mathbf{y}))}]-\gamma+h(\mathbf{y}). \label{eqn:stat_DPE}%
\end{equation}
The main facts we will need regarding this problem are the following. Under
our conditions, which include the ergodicity of the dynamics in the original
infinite swapping process $\mathbf{Y}^{\infty}$, a solution $(\gamma,W)$ to
the dynamic programming equation (\ref{eqn:stat_DPE}) exists and is unique (up
to an additive constant in $W$), and (\ref{eqn:ODEs}) defines and optimal
control. The proof of the first statement follows from classical arguments
based on approximation by so-called \textquotedblleft
discounted\textquotedblright\ control problems, and the second follows from a
verification argument very much like the one used in the last section for the
corresponding finite time problem.

Using the reversibility of the infinite swapping dynamics (discussed in
Section \ref{sec:LD}) shows that the invariant measure $\bar{\nu}$ for the
optimally controlled process is
\[
\bar{\nu}(x_{1},x_{2})=\bar{\mu}(x_{1},x_{2})e^{-2W(x_{1},x_{2})-a}%
,\ (x_{1},x_{2})\in\mathcal{S}^{2},
\]
where $a$ is a normalizing constant. However, $W$ is unique only up to an
additive constant, and so we can assume without loss that for $a=0$, $\bar
{\nu}$ defines a probability measure on $\mathcal{S}^{2}$.

\section{A diagnostic for the convergence of the empirical measure}

\label{sec:PT}

One of the challenges of Monte Carlo when dealing with problems involving rare
events is to determine when the algorithm has converged. For example, in the
setting of MCMC it can happen that the empirical measure appears to have
converged, when in reality the underlying process is stuck in some collection
of metastable states, and significant parts of the state space have not been
visited nearly often enough for a good approximation to the true equilibrium
distribution. Hence it is of interest to know if there are diagnostics that
can determine when convergence has or has not taken place.

In this section we rigorously justify a diagnostic that will tell the user
when INS has \textit{not} converged. More precisely, we will show using a
large deviations analysis that the empirical measure of the
particle/temperature association introduced in Section \ref{sec:LD} provides
such a diagnostic, as do various functionals of this empirical measure (see
Remark \ref{rem:diagnostic}). In particular, $\rho_{T}$ must converge to the
uniform distribution on $\Sigma_{K}$ if the weighted empirical measure
$\eta_{T}^{\infty}$ is to converge to the true stationary distribution.
Although the convergence of $\eta_{T}^{\infty}$ is what is needed for
computational purposes, the functionals of the empirical measure $\rho_{T}$
can be readily observed while running a simulation, and therefore provide
convenient diagnostics.

The main result of this section, Proposition \ref{prop:convTemp}, gives the
precise relation between the convergences of $\rho_{T}$ and $\eta_{T}^{\infty
}$. The result is of an asymptotic character and indicates how convergence of
the empirical measure $\eta_{T}^{\infty}$ to $\mu$ can only occur if there is
also convergence of $\rho_{T}$ to the uniform distribution. Since
$\mathcal{P}(\Sigma_{2})$ is a finite dimensional space, convergence in the
weak topology is the same as ordinary convergence as elements of a subset of a
Euclidean space. We also consider $\mathcal{P}(\mathcal{S}^{2})$ with a metric
that is consistent with weak convergence and under which it is a Polish space.
In the statement of the proposition $\mathcal{N}_{a}(w^{\ast})$ denotes the
open neighborhood about $w^{\ast}$ of radius $a$ in $\mathcal{P}(\Sigma_{2})$,
and similarly for $\mathcal{N}_{\epsilon}(\mu)\subset\mathcal{P}%
(\mathcal{S}^{2})$.

\begin{proposition}
\label{prop:convTemp} Let $w^{\ast}=(1/2,1/2)$. Then for each $a>0$ there is
an $\epsilon>0$ such that
\[
P\left(  \eta_{T}^{\infty}\in\mathcal{N}_{\epsilon}(\mu)|\rho_{T}%
\in(\mathcal{N}_{a}(w^{\ast}))^{c}\right)  \rightarrow0\ \text{as
}T\rightarrow\infty.
\]

\end{proposition}

\begin{remark}
\label{rem:diagnostic}Although the proof of Proposition \ref{prop:convTemp} is
given for the case of two temperatures, the analogous result for any finite
number of temperatures holds, though the notation needed for the proof is more
complicated. In this more general setting, it is worth noting that any
functional of $\rho_{T}$ must also converge to its asymptotic counterpart
before $\eta_{T}^{\infty}$ can converge to $\mu$. Thus one can consider
diagnostics that are based on lower dimensional quantities. For example, in
place of the $K!$-dimensional object $\rho_{T}=\{\rho_{T}^{\sigma}%
\}_{\sigma\in\Sigma_{K}}$, one could use the $K$-dimensional object defined
by
\[
\lbrack\beta_{T}]_{k}\doteq\sum_{\sigma\in\Sigma_{K}:\sigma(1)=k}\rho
_{T}^{\sigma},\quad k=1,\ldots,K.
\]
This quantity can be interpreted as follows, using the convention for particle
and temperature associations when using INS that was introduced previously.
Let each particle be associated with the temperature (dynamic) it was assigned
at time zero. Then $[\beta_{T}]_{k}$ is the fraction of time the particle
initially assigned temperature $1$ uses temperature $k$ in $[0,T]$. Since
$\rho_{T}^{\sigma}\rightarrow1/K!$ and particles are exchangeable, $[\beta
_{T}]_{k}\rightarrow1/K$ as $T\rightarrow\infty$. It is in fact this
diagnostic that has been used in previous numerical studies such as
\cite{dolplafreliudup, doldup}.
\end{remark}

To prove Proposition \ref{prop:convTemp} we first study the probability
measure that minimizes the large deviation rate $I$ for any fixed $\bar{w}%
\in\mathcal{P}(\Sigma_{K})$, with $\bar{w}$ not the uniform distribution.

\begin{lemma}
\label{lemma:minPT} For any $\bar{w}\in\mathcal{P}(\Sigma_{2})$ not equal to
the uniform distribution, the infimum of $I(\gamma,\bar{w})$ over $\gamma$ is
uniquely attained at some $\bar{\gamma}\neq\mu$.
\end{lemma}

The proof will use the ergodic control problem of the previous section. In
this section we give the proof for $K=2$, and outline the proof for general
$K$ in Remark \ref{rmk:multtemp} in the appendix. The following lemma ensures
that we can switch our focus from the unconstrained version of the
optimization problem of Lemma \ref{lemma:minPT} to a related ergodic control
problem, and that there is a correspondence between the minimizers in the two
settings. Recall the solution $(\gamma,W)$ to the Bellman equation
\eqref{eq:HJeq2} has the property that $\gamma$ is unique and $W$ is unique up
to an additive constant (see, e.g., \cite[Chapter 7]{kusdup1}).

\begin{lemma}
\label{lemma:equivOpt}Consider the optimization problem
\begin{equation}
\inf_{\nu\in\mathcal{P}(\mathcal{S}^{2})}\left\{  J(\nu)+\sum_{\mathbf{x}%
\in\mathcal{S}^{2}}h(\mathbf{x})\nu(\mathbf{x})\right\}  , \label{eq:optEq}%
\end{equation}
and also the static Bellman equation given in \eqref{eq:HJeq2} in Section
\ref{sec:controlProblem}:
\[
0=\inf_{u\in U}\left\{  \sum_{\mathbf{y}\in\mathcal{S}^{2}}\left(
u(\mathbf{y})\left[  W(\mathbf{y})-W(\mathbf{x})\right]  +r(\mathbf{x}%
,\mathbf{y})\ell\left(  \frac{u(\mathbf{y})}{r(\mathbf{x},\mathbf{y})}\right)
\right)  -\gamma+h(\mathbf{x})\right\}  .
\]
Consider any solution $(\gamma,W)$ to the Bellman equation, with $W$ taken to
be normalized in the sense that
\begin{equation}
\bar{\nu}(\mathbf{x})=\bar{\mu}(\mathbf{x})e^{-2W(\mathbf{x})}
\label{eqn:eqfornubar}%
\end{equation}
is a probability measure on $\mathbf{S}^{2}$. Then $\bar{\nu}$ is a minimizer
in \eqref{eq:optEq}. Conversely, consider a minimizing measure $\nu^{\ast}$ in
\eqref{eq:optEq}. Then a solution $(\gamma^{\ast},W^{\ast})$ to the Bellman
equation is given by
\[
W^{\ast}(\mathbf{x})=-\log\left[  \frac{d\nu^{\ast}}{d\bar{\mu}}\right]
^{1/2}(\mathbf{x}),\ \gamma^{\ast}=J(\nu^{\ast})+\sum_{\mathbf{x}%
\in\mathcal{S}^{2}}h(\mathbf{x})\nu^{\ast}(\mathbf{x}),
\]
and by uniqueness $(\gamma^{\ast},W^{\ast})=(\gamma,W)$.
\end{lemma}

\begin{proof}
We start by showing that an averaged version of the static Bellman equation
gives an upper bound for the optimization problem. Suppose $(\gamma,W)$ is a
solution to the Bellman equation and define the measure $\bar{\nu}$ by
(\ref{eqn:eqfornubar}). The Bellman equation holds with equality for all
$\mathbf{x}$, and averaging with respect to $\bar{\nu}$ gives
\begin{equation}
\gamma=\sum_{\mathbf{x},\mathbf{y}\in\mathcal{S}^{2}}r(\mathbf{x}%
,\mathbf{y})\left[  1-e^{-(W(\mathbf{y})-W(\mathbf{x}))}\right]  \bar{\mu
}(\mathbf{x})e^{-2W(\mathbf{x})}+\sum_{\mathbf{x}\in\mathcal{S}^{2}%
}h(\mathbf{x})\bar{\mu}(\mathbf{x})e^{-2W(\mathbf{x})}. \label{eqn:BellAve}%
\end{equation}
Define $\theta(\mathbf{x})=e^{-2W(\mathbf{x})}$, the likelihood ratio of
$\bar{\nu}$ and $\bar{\mu}$, and consider the rate function $J$ evaluated at
$\bar{\nu}$:
\begin{align*}
J(\bar{\nu})  &  =\sum_{\mathbf{x},\mathbf{y}\in\mathcal{S}^{2}}\left(
\theta(\mathbf{x})-\theta^{1/2}(\mathbf{x})\theta^{1/2}(\mathbf{y})\right)
r(\mathbf{x},\mathbf{y})\bar{\mu}(\mathbf{x})\\
&  =\sum_{\mathbf{x},\mathbf{y}\in\mathcal{S}^{2}}\left(  e^{-2W(\mathbf{x}%
)}-e^{-W(\mathbf{x})-W(\mathbf{y})}\right)  r(\mathbf{x},\mathbf{y})\bar{\mu
}(\mathbf{x})\\
&  =\sum_{\mathbf{x},\mathbf{y}\in\mathcal{S}^{2}}r(\mathbf{x},\mathbf{y}%
)\left(  1-e^{-(W(\mathbf{y})-W(\mathbf{x}))}\right)  \bar{\mu}(\mathbf{x}%
)e^{-2W(\mathbf{x})}.
\end{align*}
Then using (\ref{eqn:BellAve}) for the last equality,
\begin{align*}
&  \inf_{\nu\in\mathcal{P}(\mathcal{S}^{2})}\left\{  J(\nu)+\sum
_{\mathbf{x}\in\mathcal{S}^{2}}h(\mathbf{x})\nu(\mathbf{x})\right\} \\
&  \quad\leq\sum_{\mathbf{x},\mathbf{y}\in\mathcal{S}^{2}}r(\mathbf{x}%
,\mathbf{y})\left(  1-e^{-(W(\mathbf{y})-W(\mathbf{x}))}\right)  \bar{\mu
}(\mathbf{x})e^{-2W(\mathbf{x})}+\sum_{\mathbf{x}\in\mathcal{S}^{2}%
}h(\mathbf{x})e^{-2W(\mathbf{x})}\bar{\mu}(\mathbf{x})\\
&  \quad=\gamma.
\end{align*}

Next define $\gamma^{\ast}$ to be the minimal value
\[
\gamma^{\ast}=\inf_{\nu\in\mathcal{P}(\mathcal{S}^{2})}\left\{  J(\nu
)+\sum_{\mathbf{x}\in\mathcal{S}^{2}}h(\mathbf{x})\nu(\mathbf{x})\right\}  .
\]
From the previous display $\gamma^{\ast}\leq\gamma$, and we now proceed to
show the reverse inequality.

Let $\nu^{\ast}$ denote a minimizing measure, i.e.,
\[
\nu^{\ast}=\argmin_{\nu\in\mathcal{P}(\mathcal{S}^{2})}\left\{  J(\nu
)+\sum_{\mathbf{x}\in\mathcal{S}^{2}}h(\mathbf{x})\nu(\mathbf{x})\right\}  .
\]
The existence of such a measure follows from the fact that $J$ has compact
level sets and the boundedness of $h$. Moreover strict convexity of $J$
implies it is unique. Define $\theta^{\ast}=\left[  d\nu^{\ast}/d\bar{\mu
}\right]  $. For any $\mathbf{x}\in\mathcal{S}^{2}$ we have $\theta^{\ast
}(\mathbf{x})\in(0,\infty)$. For the upper bound, note that
\[
\theta^{\ast}(\mathbf{x})\leq\frac{\max_{\mathbf{x}\in\mathcal{S}^{2}}%
\nu^{\ast}(\mathbf{x})}{\min_{\mathbf{x}\in\mathcal{S}^{2}}\bar{\mu
}(\mathbf{x})}<\infty.
\]
The second inequality is due to $\nu^{\ast}$ being a probability measure and
the fact that $\bar{\mu}$ has support $\mathcal{S}^{2}$, which implies that
$\bar{\mu}(\mathbf{x})>0$ for all $\mathbf{x}\in\mathcal{S}^{2}$. It follows
from the finiteness of $\mathcal{S}^{2}$ that $\theta^{\ast}$ is bounded from
above. Moreover, by differentiating the objective function it is not difficult
to check that the optimal choice $\theta^{\ast}$ will satisfy
\[
\theta^{\ast}(\mathbf{x})=\frac{1}{4(\sum_{\mathbf{y}\in\mathcal{S}^{2}%
}r(\mathbf{x},\mathbf{y})+h(\mathbf{x}))^{2}}\left(  \sum_{\mathbf{y}%
\in\mathcal{S}^{2}}r(\mathbf{x},\mathbf{y})\left(  \theta^{\ast}\right)
^{1/2}(\mathbf{y})\right)  ^{2},\ \mathbf{x}\in\mathcal{S}^{2}.
\]
Suppose that $\theta^{\ast}(\mathbf{x})$ is zero for at least one
$\mathbf{x}\in\mathcal{S}^{2}$. Since $\sum_{\mathbf{y}\in\mathcal{S}^{2}%
}r(\mathbf{x},\mathbf{y}) + h(\mathbf{x})\in(0,\infty)$ for all $\mathbf{x}$
for $\theta^{\ast}(\mathbf{x})$ to be zero it must hold that the sum in the
last display is zero. The underlying jump rates are such that $\mathcal{S}%
^{2}$ forms a communicating class and thus for each $\mathbf{x}$ there is at
least one $\mathbf{y}$ such that $r(\mathbf{x},\mathbf{y})>0$. It follows that
$\theta^{\ast}(\mathbf{x})=0$ requires $\theta^{\ast}(\mathbf{y})=0$ for all
$\mathbf{y}$ with which $\mathbf{x}$ communicates. Repeating this argument,
using that $\mathcal{S}^{2}$ is a communicating class under the original
dynamics, shows that if $\theta^{\ast}(\mathbf{x})=0$, then $\theta^{\ast
}\equiv0$. This is clearly a contradiction and it must hold that $\theta
^{\ast}(\mathbf{x})>0$ for all $\mathbf{x}$. Hence, $\theta^{\ast}
(\mathbf{x}) \in(0,\infty)$ for all $\mathbf{x}$.

Set
\[
W^{\ast}(\mathbf{x})=-\log(\theta^{\ast})^{1/2}(\mathbf{x}).
\]
Inserting the measure $\nu^{\ast}$ into the objective function and rewriting
it in terms of $W^{\ast}$,
\begin{align*}
\gamma^{\ast}  &  =J(\nu^{\ast})+\sum_{\mathbf{x}\in\mathcal{S}^{2}%
}h(\mathbf{x})\nu^{\ast}(\mathbf{x})\\
&  =\sum_{\mathbf{x},\mathbf{y}\in\mathcal{S}^{2}}\left(  e^{-2W^{\ast
}(\mathbf{x})}-e^{-W^{\ast}(\mathbf{y})}e^{-W^{\ast}(\mathbf{x})}\right)
r(\mathbf{x},\mathbf{y})\bar{\mu}(\mathbf{x})+\sum_{\mathbf{x}\in
\mathcal{S}^{2}}h(\mathbf{x})\theta^{\ast}(\mathbf{x})\bar{\mu}(\mathbf{x})\\
&  =\sum_{\mathbf{x},\mathbf{y}\in\mathcal{S}^{2}}r(\mathbf{x},\mathbf{y}%
)\left(  1-e^{-(W^{\ast}(\mathbf{y})-W^{\ast}(\mathbf{x}))}\right)  \bar{\mu
}(\mathbf{x})e^{-2W^{\ast}(\mathbf{x})}+\sum_{\mathbf{x}\in\mathcal{S}^{2}%
}h(\mathbf{x})\bar{\mu}(\mathbf{x})e^{-2W^{\ast}(\mathbf{x})}.
\end{align*}
This is the Bellman equation averaged with respect to $\bar{\mu}%
(\mathbf{x})e^{-2W^{\ast}(\mathbf{x})}$.

The cost (average cost per unit time) associated with the control $u^{\ast}$
is (see Section \ref{sec:controlProblem})
\[
\lim_{T\rightarrow\infty}E\left[  \frac{1}{T}\int_{0}^{T}\left(
\sum_{\boldsymbol{\mathbf{y}}\in\mathcal{S}^{2}}r(\bar{\mathbf{Y}}^{\infty
}(s),\boldsymbol{\mathbf{y}})\ell\left(  \frac{u^{\ast}(s,\bar{\mathbf{Y}%
}^{\infty}(s);\boldsymbol{\mathbf{y}})}{r(\bar{\mathbf{Y}}^{\infty
}(s),\boldsymbol{\mathbf{y}})}\right)  +h(\bar{\mathbf{Y}}^{\infty
}(s))\right)  ds\right]  .
\]
where the controlled process $\bar{\mathbf{Y}}^{\infty}$ has dynamics
according to the choice of control $u^{\ast}$. Using the same calculations as
in Section \ref{sec:controlProblem}, the invariant measure associated with
this process is precisely $\nu^{\ast}$ and by ergodicity the limit in the last
display is the average of the cost with respect to $\nu^{\ast}$:
\begin{align*}
&  \sum_{\mathbf{x},\mathbf{y}\in\mathcal{S}^{2}}\left(  r(\mathbf{x}%
,\mathbf{y})\ell\left(  \frac{u^{\ast}(\mathbf{y})}{r(\mathbf{x},\mathbf{y}%
)}\right)  +h(\mathbf{x})\right)  \nu^{\ast}(\mathbf{x})\\
&  \quad=\sum_{\mathbf{x},\mathbf{y}\in\mathcal{S}^{2}}\left(  r(\mathbf{x}%
,\mathbf{y})\left[  1-e^{-W^{\ast}(\mathbf{y})+W^{\ast}(\mathbf{x})}\right]
+h(\mathbf{x})-r(\mathbf{x},\mathbf{y})\left(  W^{\ast}(\mathbf{y})-W^{\ast
}(\mathbf{x})\right)  \right)  \nu^{\ast}(\mathbf{x})\\
&  \quad=\sum_{\mathbf{x},\mathbf{y}\in\mathcal{S}^{2}}\left(  r(\mathbf{x}%
,\mathbf{y})\left[  1-e^{-W^{\ast}(\mathbf{y})+W^{\ast}(\mathbf{x})}\right]
+h(\mathbf{x})\right)  \nu^{\ast}(\mathbf{x})-\sum_{\mathbf{x}\in
\mathcal{S}^{2}}\mathcal{L}^{u^{\ast}}W^{\ast}(\mathbf{x})\nu^{\ast
}(\mathbf{x}).
\end{align*}
The first term is $\gamma^{\ast}$. Moreover, since $\mathcal{L}^{u^{\ast}}$ is
the generator associated with $\nu^{\ast}$, the second term is $0$ by
Echeverria's theorem \cite{ethkur}. Thus, $\gamma^{\ast}$ is also the cost
obtained using the control $u^{\ast}$. Since $\gamma$ is the optimal cost it
follows that $\gamma^{\ast}\geq\gamma$.

Combining the two inequalities gives $\gamma=\gamma^{\ast}$. This implies that
the two measures $\bar{\nu}$ and $\nu^{\ast}$ are both minimizers in
\eqref{eq:optEq}. Strict convexity of the rate function $J$ then ensures that
the two measures are in fact the same, $\bar{\nu}=\nu^{\ast}$ and by extension
$W=W^{\ast}$.
\end{proof}

\vspace{\baselineskip}

\begin{proof}
[Proof of Lemma \ref{lemma:minPT}]\medskip Fix a $\bar{w}\in\mathcal{P}%
(\Sigma_{2})\setminus\{(1/2,1/2)\}$ and consider the optimization problem
\begin{align}
\inf\{I(\gamma,\bar{w})  &  :\ \gamma\in\mathcal{P}(\mathcal{S}^{2}%
)\}\label{eq:optProblem}\\
&  =\inf_{\nu\in\mathcal{P}(\mathcal{S}^{2})}\left\{  J(\nu):\ \sum
_{\mathbf{x}\in\mathcal{S}^{2}}\rho(\mathbf{x})\nu(\mathbf{x})=\bar{w}%
_{1},\ \sum_{\mathbf{x}\in\mathcal{S}^{2}}\rho(\mathbf{x}^{R})\nu
(\mathbf{x})=\bar{w}_{2}\right\}  .\nonumber
\end{align}
We are only interested in the minimizing measure $\nu$ and it is enough to
consider any optimization problem that will have the same minimizer as
\eqref{eq:optProblem}. Using Lagrange multipliers $\lambda_{1},\lambda_{2}$,
\eqref{eq:optProblem} can be formulated as the unconstrained optimization
problem
\[
\min_{\gamma,w}\ \left\{  I(\gamma,w)+\lambda_{1}(w_{1}-\bar{w}_{1}%
)+\lambda_{2}(w_{2}-\bar{w}_{2})\right\}  ;
\]
see, e.g., Theorem 8.1 and its extension to equality constraints in \cite{lue}
for the existence of multipliers $\lambda_{1},\lambda_{2}$. Using the fact
that necessarily $\bar{w}_{1}+\bar{w}_{2}=w_{1}+w_{2}=1$, this has the same
minimizer as $\min_{\gamma,w}\ \left\{  I(\gamma,w)+(\lambda_{1}-\lambda
_{2})(w_{1}-\bar{w}_{1})\right\}  $, where the multipliers are chosen so that
$w_{1}=\bar{w}_{1}$. With such multipliers given (and fixed), by using the
definition of $I(\gamma,w)$ the optimization problem becomes
\[
\min_{\nu\in\mathcal{P}(\mathcal{S}^{2})}\left\{  J(\nu)+(\lambda_{1}%
-\lambda_{2})\sum_{\mathbf{x}\in\mathcal{S}^{2}}\rho(\mathbf{x})\nu
(\mathbf{x})\right\}  -(\lambda_{1}-\lambda_{2})\bar{w}_{1},
\]
and we further simplify by dropping the term $-(\lambda_{1}-\lambda_{2}%
)\bar{w}_{1}$.

The Lagrange multipliers $\lambda_{1}$ and $\lambda_{2}$ correspond to
$\bar{w}\neq(1/2,1/2)$ so it cannot be the case that $\lambda_{1}=\lambda_{2}%
$. Let $\lambda$ denote the difference $\lambda_{1}-\lambda_{2}$; without loss
of generality we can assume that $\lambda>0$. Thus, in order to prove the
claim it is enough to consider the minimizer of
\begin{equation}
\min_{\nu\in\mathcal{P}(\mathcal{S}^{2})}\left\{  J(\nu)+\sum_{\mathbf{x}%
\in\mathcal{S}^{2}}h(\mathbf{x})\nu(\mathbf{x})\right\}  , \label{eq:min1}%
\end{equation}
where $h(\mathbf{x})=\lambda\rho(\mathbf{x})$.

For any cost function $h$ we can define the functional $F:\mathcal{P}%
(\mathcal{S}^{2})\rightarrow\mathbb{R}$ by $F(\nu)\doteq\sum_{\mathbf{x}%
\in\mathcal{S}^{2}}h(\mathbf{x})\nu(\mathbf{x})$. This choice of $F$ is of the
form considered in Section \ref{sec:controlProblem}. In particular, $F$ is
bounded and continuous, and from the Laplace principle for $\nu_{T}$ it
follows that
\begin{align*}
\lim_{T\rightarrow\infty}-\frac{1}{T}\log E\left[  e^{-TF(\nu_{T})}\right]
&  =\inf_{\nu\in\mathcal{P}(\mathcal{S}^{2})}\left\{  J(\nu)+F(\nu)\right\} \\
&  =\inf_{\nu\in\mathcal{P}(\mathcal{S}^{2})}\left\{  J(\nu)+\lambda
\sum_{\mathbf{x}\in\mathcal{S}^{2}}\rho(\mathbf{x})\nu(\mathbf{x})\right\}  .
\end{align*}
In Lemma \ref{lemma:equivOpt} it was shown that the minimizer of
\eqref{eq:min1} is
\[
\bar{\nu}(\mathbf{x})=\bar{\mu}(\mathbf{x})e^{-2W(\mathbf{x})},
\]
where $(\gamma,W)$ is a solution to the Bellman equation%
\begin{equation}
0=\sum_{\mathbf{y}\in\mathcal{S}^{2}}r(\mathbf{x},\mathbf{y})\left[
1-e^{-W(\mathbf{x})+W(\mathbf{y})}\right]  -\gamma+\lambda\rho(\mathbf{x}%
),\ \mathbf{x}\in\mathcal{S}^{2}, \label{eqn:Bellman}%
\end{equation}
with the value function $W$ normalized to make $\bar{\nu}$ a probability measure.

Next we use the fact that for this particular choice of $h(\mathbf{x})$,
$\gamma<\lambda/2$ (see Lemma \ref{lem:cost_assym} in the Appendix). The
intuition here is that when $\rho\equiv1/2$ there is no incentive to use the
control and we get the cost $\gamma=\lambda/2$. Moreover when $\rho$ depends
on $\mathbf{x}$ not using an active control will result in the same cost,
since $\rho(\mathbf{x})+\rho(\mathbf{x}^{R})=1$ for all $\mathbf{x}$. In
contrast, by accepting a small increase in the cost due to active control we
can lower the running cost substantially by favoring states with lower $\rho$
value than their symmetric counterpart.

The aim is to show that the optimal measure $\bar{\nu}$ is such that
$M\bar{\nu}\neq\mu$. From the definitions of $M$ and $\rho$ [given for two
temperatures in \eqref{eqn:rho} and \eqref{eq:mapM}, respectively] for
$\mathbf{x}\in\mathcal{S}^{2}$
\begin{align*}
(M\bar{\nu})(\mathbf{x})  &  =\rho(\mathbf{x})\left(  \bar{\nu}(\mathbf{x}%
)+\bar{\nu}(\mathbf{x}^{R})\right) \\
&  =\rho(\mathbf{x})\left(  \bar{\mu}(\mathbf{x})e^{-2W(\mathbf{x})}+\bar{\mu
}(\mathbf{x}^{R})e^{-2W(\mathbf{x}^{R})}\right) \\
&  =\mu(\mathbf{x})\frac{e^{-2W(\mathbf{x})}+e^{-2W(\mathbf{x}^{R})}}{2}.
\end{align*}
We argue by contradiction. Suppose that $M\bar{\nu}=\mu$. Then by the last
display
\begin{equation}
e^{-2W(\mathbf{x})}+e^{-2W(\mathbf{x}^{R})}-2=0\ \text{for}\ \mathbf{x}%
\in\mathcal{S}^{2}. \label{eq:condW}%
\end{equation}
Let $\mathcal{D}$ denote the set of diagonal states in $\mathcal{S}^{2}$:
\[
\mathcal{D}\doteq\{\mathbf{x}\in\mathcal{S}^{2}:\ \mathbf{x}=\mathbf{x}%
^{R}\}.
\]
Note that for $\mathbf{x}\in\mathcal{D}$ always $\rho(\mathbf{x})=1/2$ and
hence $h(\mathbf{x})=\lambda/2$. Moreover, as a special case of
\eqref{eq:condW}, the value function $W$ must be zero on the diagonal
$\mathcal{D}$.

The symmetrized dynamics of the INS process (see Lemma \ref{lemma:symmetry})
imply that if $\mathbf{x}\in\mathcal{D}$ communicates directly with a state
$\mathbf{y}$ then it communicates directly with $\mathbf{y}^{R}$ as well. To
simplify notation we therefore let $\mathcal{A}$ denote the collection of
states that lie above the diagonal. For states $\mathbf{x}\in\mathcal{D}$ the
Bellman equation (\ref{eqn:Bellman}) then takes the form
\[
0=\sum_{\mathbf{y}\in\mathcal{A}}\left(  r(\mathbf{x},\mathbf{y})\left[
1-e^{-(W(\mathbf{y})-W(\mathbf{x}))}\right]  +r(\mathbf{x},\mathbf{y}%
^{R})\left[  1-e^{-(W(\mathbf{y}^{R})-W(\mathbf{x}))}\right]  \right)
-\gamma+\frac{\lambda}{2}.
\]
By symmetry of the rates $r$, if $\mathbf{x}\in\mathcal{D}$ then
$r(\mathbf{x},\mathbf{y})=r(\mathbf{x},\mathbf{y}^{R})$. Combined with the
constraint that $W$ is zero on $\mathcal{D}$ we can rewrite the Bellman
equation for $\mathbf{x}\in\mathcal{D}$ as
\[
0=\sum_{\mathbf{y}\in\mathcal{A}}r(\mathbf{x},\mathbf{y})\left(
2-e^{-W(\mathbf{y})}-e^{-W(\mathbf{y}^{R})}\right)  -\gamma+\frac{\lambda}%
{2}.
\]

The assumptions on $\tau_{1},\tau_{2}$ and the potential $V$ ensure that
$\rho(\mathbf{x})$ is not identically equal to $1/2$, and therefore by Lemma
\ref{lem:cost_assym} in the Appendix $\gamma<\lambda/2$. It follows that $W$
satisfies
\begin{equation}
\sum_{\mathbf{y}\in\mathcal{A}}r(\mathbf{x},\mathbf{y})\left(
2-e^{-W(\mathbf{y})}-e^{-W(\mathbf{y}^{R})}\right)  =\gamma-\frac{\lambda}%
{2}<0,\ \mathbf{x}\in\mathcal{D}. \label{eq:ineqBellman}%
\end{equation}
The rates $r(\mathbf{x},\mathbf{y})$ are all nonnegative and for
\eqref{eq:ineqBellman} to hold requires that for at least one of the states
$\mathbf{y}\in\mathcal{A}$
\[
e^{-W(\mathbf{y})}+e^{-W(\mathbf{y}^{R})}>2.
\]
However, this inequality is not compatible with (\ref{eq:condW}) [see Lemma
\ref{lemma:ineqCont} in the Appendix with $K=2$, $a_{1}=e^{-W(\mathbf{y})}$
and $a_{2}=e^{-W(\mathbf{y}^{R})}$]. Hence, the condition \eqref{eq:condW}
violates the Bellman equation for diagonal states and it cannot be that
$M\bar{\nu}=\mu$. This completes the proof.\medskip
\end{proof}

\begin{proof}
[Proof of Proposition \ref{prop:convTemp}]In addition to showing the claimed
convergence as $T\rightarrow\infty$ we will show that the probability decays
exponentially in $T$. Consider the mapping
\[
a\rightarrow R(a)\doteq\inf\left\{  I(\gamma,w):w\in\left(  \mathcal{N}%
_{a}(w^{\ast})^{c}\right)  \right\}  .
\]
The rate function $J$ has $\bar{\mu}$, the symmetrized version of the original
stationary distribution, as its unique minimizer and $J(\bar{\mu})=0$.
Moreover, $\bar{\mu}$ maps to $w^{\ast}$ in the sense that $\langle
\rho^{\sigma},\bar{\mu}\rangle=1/2$ for $\sigma=\{1,2\}$ and $\sigma=\{2,1\}$
(the permutations available for two temperatures); see Section \ref{sec:LD}.
Consider the set
\[
C_{a}\doteq\left\{  (\gamma,w):w\in\left(  \mathcal{N}_{a}(w^{\ast})\right)
^{c}\right\}  .
\]
This is a closed set and, since $I$ is a rate function on $\mathcal{P}%
(\mathcal{S}^{2})\times\mathcal{P}(\Sigma_{2})$, the infimum of $I$ over
$C_{a}$ is achieved and necessarily $R(a)>0$ whenever $a>0$. The mapping
$a\rightarrow R(a)$ is monotone and thus continuous on a dense subset of
$(0,1)$. Therefore, without loss of generality we may assume $a$ to be a
continuity point of $R$ [if not, just replace $a$ by a continuity point in
$(0,a)$]. Using the definition \eqref{eq:I2} of $I_{2}$ it holds that
\[
\inf_{w\in(\mathcal{N}_{a}(w^{\ast}))^{c}}I_{2}(w)=\inf_{w\in\left(
(\mathcal{N}_{a}(w^{\ast}))^{c}\right)  ^{\circ}}I_{2}(w)=R(a),
\]
and
\[
\lim_{T\rightarrow\infty}\frac{1}{T}\log P(\rho_{T}\in(\mathcal{N}_{a}%
(w^{\ast}))^{c})=-R(a).
\]

Next, for some $\epsilon>0$, consider the event
\[
\left\{  \eta_{T}^{\infty}\in\mathcal{N}_{\epsilon}(\mu),\ \rho_{T}%
\in(\mathcal{N}_{a}(w^{\ast}))^{c}\right\}  .
\]
Using the large deviation upper bound,
\begin{align*}
&  \limsup_{T\rightarrow\infty}\frac{1}{T}\log P\left(  \eta_{T}^{\infty}%
\in\mathcal{N}_{\epsilon}(\mu),\ \rho_{T}\in(\mathcal{N}_{a}(w^{\ast}%
))^{c}\right) \\
&  \quad\leq-\inf\left\{  I(\gamma,w):\gamma\in\bar{\mathcal{N}}_{\epsilon
}(\mu),\ w\in(\mathcal{N}_{a}(w^{\ast}))^{c}\right\}  ,
\end{align*}
where $\bar{\mathcal{N}}_{\epsilon}(\mu)$ is the closure of ${\mathcal{N}%
}_{\epsilon}(\mu)$. We now claim that for small enough $\epsilon>0$ the
infimum in the last display is strictly larger than $R(a)$. If so, then%
\begin{align*}
&  \limsup_{T\rightarrow\infty}\frac{1}{T}\log P\left(  \eta_{T}^{\infty}%
\in\mathcal{N}_{\epsilon}(\mu)|\rho_{T}\in(\mathcal{N}_{a}(w^{\ast}%
))^{c}\right) \\
&  \quad\leq-\inf\left\{  I(\gamma,w):\gamma\in\bar{\mathcal{N}}_{\epsilon
}(\mu),\ w\in(\mathcal{N}_{a}(w^{\ast}))^{c}\right\}  +R(a)\\
&  \quad<0,
\end{align*}
which gives the exponential decay to zero of the conditional probability.

To show that the infimum is greater than $R(a)$ for $\epsilon>0$ small enough
we argue by contradiction. For the given $a$ the set
\begin{align*}
&  \{(\gamma,w):I(\gamma,w)=R(a),w\in(\mathcal{N}_{a}(w^{\ast}))^{c}\}\\
&  \quad=\{(\gamma,w):I(\gamma,w)\leq R(a),w\in(\mathcal{N}_{a}(w^{\ast}%
))^{c}\}
\end{align*}
is compact. In addition, by Lemma \ref{lemma:minPT} the projection onto the
first component does not contain the original invariant measure $\mu$. If the
claim is not true then for every $\epsilon=1/n$, $n\in\mathbb{N}$,
\[
\inf\{I(\gamma,w):\gamma\in\bar{\mathcal{N}}_{1/n}(\mu),w\in(\mathcal{N}%
_{a}(w^{\ast}))^{c}\}\leq R(a).
\]
Since $I$ has compact level sets this means the infimum is attained, and
hence, using compactness of level sets once more, by choosing a convergent
subsequence (indexed by $n$) we have $(\gamma_{n},w_{n})$ such that
\[
I(\gamma_{n},w_{n})\leq R(a),\ \gamma_{n}\rightarrow\mu,\ \text{and }%
w_{n}\rightarrow\hat{w}\in(\mathcal{N}_{a}(w^{\ast}))^{c}.
\]
By lower semicontinuity $I(\mu,\hat{w})\leq R(a)<\infty$. However, this
contradicts Lemma \ref{lemma:minPT}, and completes the proof.
\end{proof}

\section{Further qualitative properties}

\label{sec:asymmetry} In this section we study other qualitative properties of
the infinite swapping process and, by extension, parallel tempering. In
particular, we consider the question of how symmetries and asymmetries of the
energy landscape can affect the behavior of infinite swapping. As we will see,
when an energy landscape is symmetric in the sense that the energy potential
at the minima of the two wells are the same, then subject to the condition
that the higher temperature is sufficiently large that the energy barrier is
not an obstacle to movement between the wells at that temperature, infinite
swapping will converge rapidly. However, when asymmetry holds and the values
at the local minima are not the same, then surprising behavior can result. In
fact, a \textquotedblleft secondary metastability\textquotedblright\ can
emerge, depending on the degree to which the depths of the two wells are not
symmetric. This second metastability issue is less of a hindrance than the
original energy barrier, but could substantially slow convergence of the
weighted empirical measure. In fact a counter-intuitive behavior is observed,
in that even \textit{reducing} the energy barrier of one well while holding
the other constant may slow the convergence of $\eta_{T}^{\infty}$. Besides a
heuristic explanation for this, we will demonstrate the effect via the large
deviation rate function using the stochastic control interpretation of the
last section.

A second issue we discuss is related to the fact that optimizers in a large
deviations analysis can explain how a particular rare event occurs. Suppose
one observes that the weighted empirical measure has not properly assigned
mass between the wells in such a two well model. By minimizing the rate
function subject to such a constraint (e.g., the well to the right is given a
fraction $\kappa(1-\delta)$ of the mass under $\eta_{T}^{\infty}$ when the
stationary measure gives it $\kappa$), one can find the most likely observed
weighted empirical distribution given this \textquotedblleft
error.\textquotedblright\ The solution to the associated stochastic control
problem, and specifically the form of the feedback control, will then identify
how this error occurred. In particular, it identifies those places in the
state space where poor sampling of the underlying distribution has the largest
impact and produces the greatest error. This issue is discussed and
illustrated via numerical computation at the end of the section.

\subsection{Symmetric and asymmetric double wells}

In this subsection and the next the potential will have two local minima at
$x_{L}<0$ and $x_{R}>0$ and a local maximum (top of the separating barrier) at
$0$. For simplicity we think of the underlying state space $\mathcal{S}$ as
being a grid in $\mathbb{R}$ that includes the local minima as well as the top
of the barrier. The restriction to $\mathbb{R}$ is to simplify the discussion,
while the assumption of a finite state space will allow the explicit numerical
solution to certain optimization problems. However, the conclusions will hold
more generally within the restriction of a two-well landscape.

We first review certain properties of the INS process. The INS process
$\mathbf{Y}$, which takes values in $\mathbb{R}^{2}$, has four stable points,
which one can view as being the local minima of an implied cost potential of
the form
\[
U(y_{1},y_{2})=-\log\left(  e^{-\frac{1}{\tau_{1}}V(y_{1})-\frac{1}{\tau_{2}%
}V(y_{2})}+e^{-\frac{1}{\tau_{1}}V(y_{2})-\frac{1}{\tau_{2}}V(y_{1})}\right)
.
\]
Note that regardless of the form of $V$, this potential is symmetric about the
diagonal $y_{1}=y_{2}$, and hence the dynamics are likewise symmetric. The
mean value for increments of the INS process are illustrated in Figure
\ref{fig:1}.
\begin{figure}
[ptb]
\begin{center}
\includegraphics[
height=2.597in,
width=2.8141in
]%
{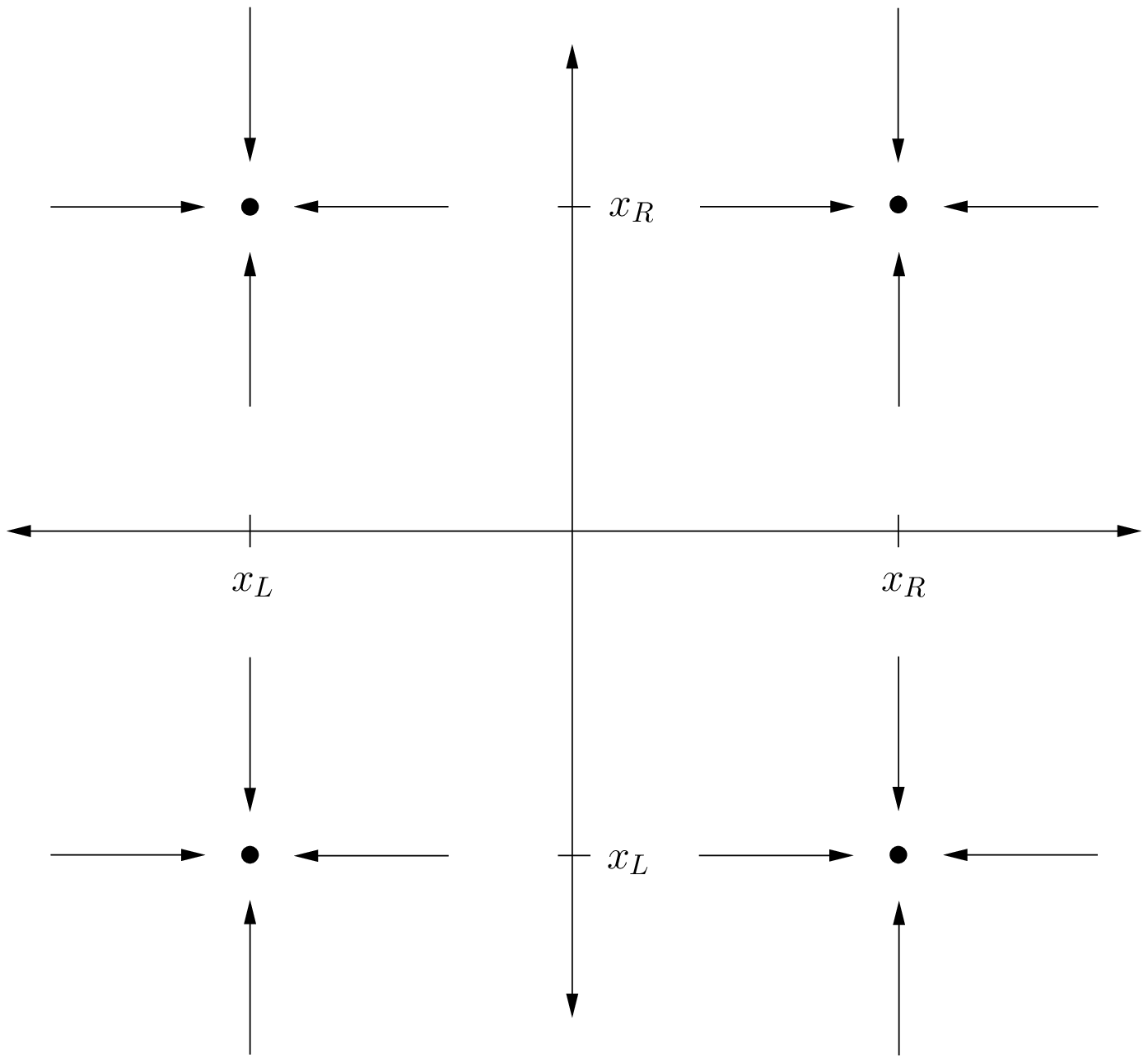}%
\label{fig:1}%
\end{center}
\end{figure}
If one considers $\mathbf{Y}$ as a process in $\mathbb{R}^{2}$ with the
indicated metastable states, then the primary impediment to good sampling and
an accurate approximation to the symmetrized distribution $\frac{1}{2}%
[\mu(x_{1},x_{2})+\mu(x_{2},x_{1})]$ is the movement of the process between
neighborhoods of the four states $(x_{L},x_{L}),(x_{L},x_{R}),(x_{R},x_{L})$
and $(x_{R},x_{R})$.

An alternative perspective is to consider $\mathbf{Y}$ as giving the locations
of two particles, whose transition rates are described by (\ref{eq:defGamma}),
with $\rho$ from (\ref{eqn:rho}) of the form
\[
\rho(x_{1},x_{2})=\frac{1}{1+e^{\left[  \frac{1}{\tau_{1}}-\frac{1}{\tau_{2}%
}\right]  (V(x_{1})-V(x_{2}))}}.
\]
Since $\Gamma^{1}$ is the intensity matrix of the low temperature dynamics and
$\Gamma^{2}$ that of the high temperature, we see that whenever $V(Y_{1})$ is
larger than $V(Y_{2})$ by a certain amount, then $Y_{1}$ has essentially been
given the high temperature dynamics and $Y_{2}$ the low, and conversely. This
is due to the exponential scaling in $\rho$, and has an analogue for parallel
tempering when the rate of swap attempts is high. The situation is illustrated
in Figure \ref{fig:rates} for an asymmetric landscape.
\begin{figure}
[ptb]
\begin{center}
\includegraphics[
height=1.1173in,
width=4.3007in
]%
{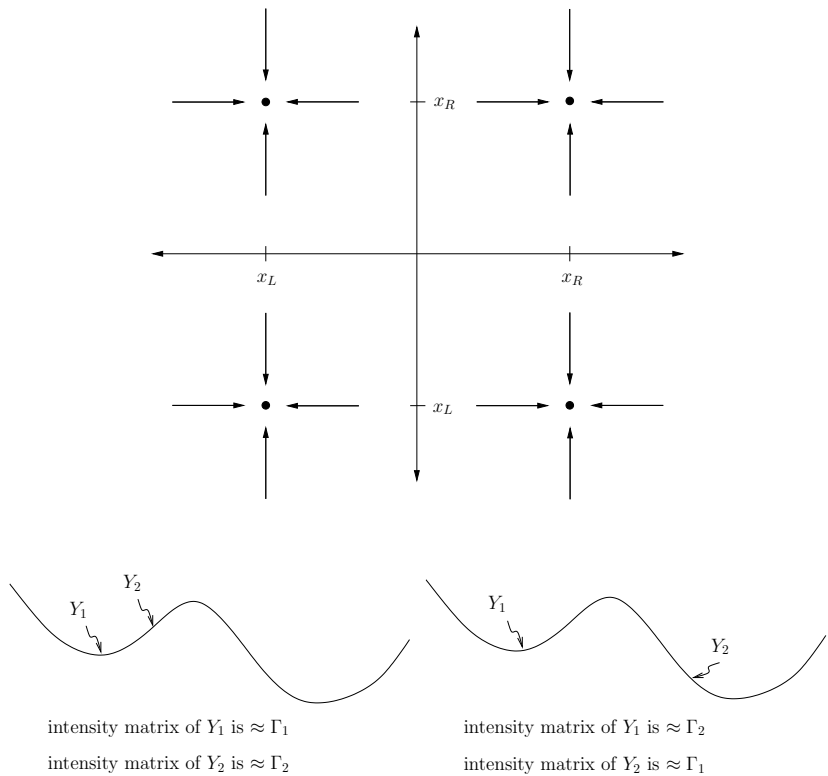}%
\caption{Assignment of dynamics due to relative heights}%
\label{fig:rates}%
\end{center}
\end{figure}

We next consider how this \textquotedblleft toggling\textquotedblright%
\ between high and low temperature dynamics affects the behavior of
$\mathbf{Y}$, and in particular how it affects the qualitative properties of
the empirical measure of $\mathbf{Y}$ with regard to sampling in
$\mathbb{R}^{2}$. Suppose the well is symmetric as in Figure \ref{fig:sym}
with wells of depth $h$.
\begin{figure}
[ptb]
\begin{center}
\includegraphics[
height=2.1223in,
width=2.9594in
]%
{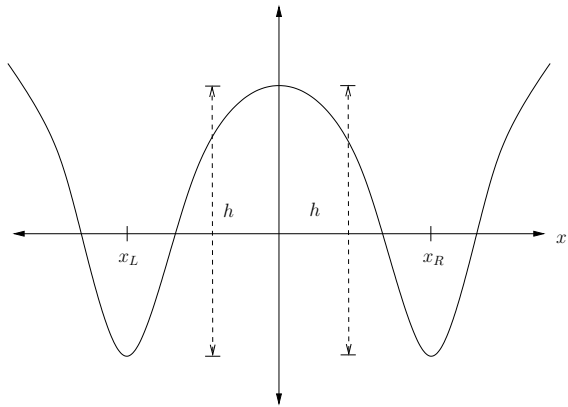}%
\caption{Symmetric double well}%
\label{fig:sym}%
\end{center}
\end{figure}
Recall the assumption that the higher temperature is such that the process can
easily cross the barrier separating the two wells. We claim that this implies
the infinite swapping process $\mathbf{Y}$ easily moves between the four
metastable points. Indeed, if both particles are placed in the left well [so
that $\mathbf{Y}$ is near $(x_{L},x_{L})$] then after a relatively short time
one of the two particles will scale the landscape to a point somewhere between
$x_{L}$ and $0$, while the other is near $x_{L}$. At this stage the particle
higher on the energy landscape essentially follows the dynamic of the
intensity matrix $\Gamma_{2}$, at least until it and the particle near $x_{L}$
reach places in the landscape of the same height. This can happen by the
particle falling back to the basin of the left well, or by crossing the
barrier to enter the neighborhood of $x_{R}$. By assumption, the latter is not
a particularly rare event, and corresponds to $\mathbf{Y}$ ending up near
$(x_{L},x_{R})$ or $(x_{R},x_{L})$. The argument can be repeated and, owing to
the symmetry, movement of $\mathbf{Y}$ between neighborhoods of all the points
$(x_{L},x_{L}),(x_{L},x_{R}),(x_{R},x_{L})$ and $(x_{R},x_{R})$ occurs with
the same frequency.

We next contrast this with what can be expected when the heights are changed,
and in particular if there is a reduction of the height of the right well, so
that in Figure \ref{fig:asym} we have $h_{L}=h$ and $h_{R}<h$. In the case the
discussion just given applies equally well when both particles start in a
neighborhood of $x_{L}$, but the behavior of the process now changes radically
when $\mathbf{Y}$ is near $(x_{L},x_{R})$ or $(x_{R},x_{L})$. To be specific,
assume that it is near $(x_{L},x_{R})$. The transitions of interest are: (a)
$Y_{1}$ joins $Y_{2}$ in the right well, and (b) $Y_{2}$ returns and joins
$Y_{1}$ in the left well. The event (b) is actually quite likely, since if
$Y_{1}$ is in the deeper well then it is highly probable that it is lower on
the energy landscape, and therefore $Y_{2}$ is given the high temperature
dynamic. For the same reason (a) is unlikely. Indeed, the only way it can
happen is if $Y_{1}$, in spite of being given the lower temperature dynamics,
is able to move up the landscape to a point were it exceeds the typical energy
value that $Y_{2}$ sees while in the right well (and using the higher
temperature dynamic). Thus $Y_{1}$ must overcome an energy barrier, whose size
is related to the degree of asymmetry of the two well depths. We call this a
\textit{secondary metastability,} and note that the effect of lowering one of
the energy barriers in the single particle model is that it increases an
energy barrier for the two particle INS process, leading to poorer sampling of
the state space by the process.%
\begin{figure}
[ptb]
\begin{center}
\includegraphics[
height=2.1223in,
width=2.981in
]%
{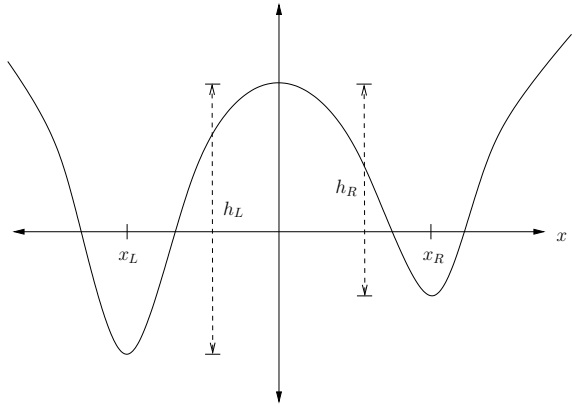}%
\caption{Asymmetric double well}%
\label{fig:asym}%
\end{center}
\end{figure}

\subsection{Numerical example}

The issue described in the last section is reflected in the large deviation
rate function. This will be illustrated by numerically solving a constrained
optimization problem, though with a constraint of a different form from the
last section. We are particularly interested in the impact of the secondary
metastability on the accuracy of integrals with respect to the low temperature
marginal. Specifically, we consider%
\[
\inf_{\nu\in\mathcal{P}(\mathcal{S}^{2})}\left\{  J(\nu):\ \sum_{\mathbf{x}%
\in\lbrack0,\infty)\times\mathcal{S}}\rho(\mathbf{x})\nu(\mathbf{x}%
)+\rho(\mathbf{x}^{R})\nu(\mathbf{x}^{R})=\kappa(1-\delta)\right\}  ,
\]
where as before $\mathbf{x}^{R}=(x_{2},x_{1})$. Here $\kappa\in(0,1)$ is the
mass that the low temperature marginal places on the set $[0,\infty)$, and
$\delta$ is the size of the error. Recalling that $\rho(\mathbf{x}%
)\nu(\mathbf{x})+\rho(\mathbf{x}^{R})\nu(\mathbf{x}^{R})$ is the mapping that
takes a symmetrized measure to its weighted counterpart, this variational
problem will identify the most likely distribution for the mass given that it
is incorrectly assigned by the sampling. Note that one expects the relative
distribution within each well to converge much faster than the relative
weights between wells. Hence this constraint focuses attention on the most
likely error that the sampling must overcome, which is properly assigning the
mass between the two wells.

We study this problem using the same methods as in the previous section, i.e.,
introduce a Lagrange multiplier and analyze the resulting ergodic control
problem. Since the only difference is the form of the constraint, we do not
repeat any details in the derivation of the control problem. For numerical
illustration of the effect of asymmetry in the potential landscape we use the
form of Glauber dynamics for two temperatures as defined in Section
\ref{sec:Model}. The claim of the previous subsection, based on the heuristic
discussion involving the stochastic control problem of Section
\ref{sec:controlProblem}, is that the secondary metastability induced by an
asymmetric potential $V$ slows the convergence of $\eta_{T}^{\infty}$ to $\mu
$. To demonstrate this effect, we show that the rate associated with the same
relative error (i.e., value of $\delta$) is lower for the asymmetric case,
indicating that the outcome is more likely.

In order to have a convenient way of constructing potentials with similar
shape but varying degree of asymmetry we use the following family of
functions, sometimes referred to as Franz potential,
\[
V(x)=V(x;\alpha)=\frac{3x^{4}-4(\alpha-1)x^{3}-6\alpha x^{2}}{2\alpha+1}+1,
\]
For every value of $\alpha$, $V(\cdot;\alpha)$ has a fixed local minimum at
$x_{L}=-1$, a varying local minimum at $x_{R}=\alpha$ and a fixed barrier of
height $1$ at the origin. Figure \ref{fig:Franz2} shows the potential $V$ for
some values of $\alpha$. In particular, taking $\alpha=1$ produces a symmetric
two well potential and $\alpha=0$ produces a single well.
\begin{figure}
[ptb]
\begin{center}
\includegraphics[
height=1.5601in,
width=2.4768in
]%
{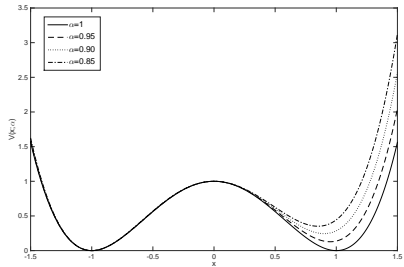}%
\caption{Franz potential $V$ for some values of $\alpha$.}%
\label{fig:Franz2}%
\end{center}
\end{figure}

Table \ref{table:kappa} shows the value of $\kappa=\mu_{1}([0,\infty))$ for
some values of $\alpha$ when the Franz potential is used to define the
underlying Gibbs measure.

\begin{table}[h]
\centering%
\begin{tabular}
[c]{|c|c|c|c|c|c|}\hline\hline
$\alpha$ & 1 & 0.97 & 0.95 & 0.90 & 0.85\\\hline
$\kappa$ & 0.500 & 0.318 & 0.223 & 0.0840 & 0.0316\\\hline\hline
\end{tabular}
\caption{The probability $\kappa=\mu_{1}([0,\infty))$ for some values of
$\alpha$ in the Franz potential}%
\label{table:kappa}%
\end{table}In Table \ref{table:numFrantz} numerical results for the discussed
optimization are presented. Table \ref{table:numFrantz2} repeats the results
but now normalized to the symmetric case $\alpha=1$ for each $\delta$.
\begin{table}[h]
\centering%
\begin{tabular}
[c]{|c|c|c|c|c|c|}\hline\hline
& \multicolumn{5}{c|}{$\alpha$}\\\hline
$\delta$ & 1 & 0.97 & 0.95 & 0.90 & 0.85\\\hline
$0.05$ & 1.5250e-5 & 8.5478e-6 & 6.0461e-6 & 2.7959e-6 & 1.4012e-6\\
$0.10$ & 6.1151e-5 & 3.4911e-5 & 2.4887e-5 & 1.1609e-5 & 5.8911e-6\\
$0.15$ & 1.3802e-4 & 8.0513e-5 & 5.7975e-5 & 2.7562e-5 & 1.4163e-5\\
$0.20$ & 2.4655e-4 & 1.4704e-4 & 1.0702e-4 & 5.1900e-5 &
2.7206e-5\\\hline\hline
\end{tabular}
\caption{Large deviation rate for the minimizing measure with a
low-temperature marginal that puts mass $\kappa(1-\delta)$ in the shallow
well, for different values of $\alpha$ in the Franz potential $V$; $\tau
_{1}=0.1,\ \tau_{2}=0.5$, $|\mathcal{S}|=12$.}%
\label{table:numFrantz}%
\end{table}\begin{table}[h]
\centering%
\begin{tabular}
[c]{|c|c|c|c|c|c|}\hline\hline
& \multicolumn{5}{c|}{$\alpha$}\\\hline
$\delta$ & 1 & 0.97 & 0.95 & 0.90 & 0.85\\\hline
$0.05$ & 1 & 0.5605 & 0.3965 & 0.1833 & 0.09188\\
$0.10$ & 1 & 0.5709 & 0.4070 & 0.1898 & 0.09634\\
$0.15$ & 1 & 0.5833 & 0.4200 & 0.1997 & 0.1026\\
$0.20$ & 1 & 0.5964 & 0.4341 & 0.2105 & 0.1103\\\hline\hline
\end{tabular}
\caption{Large deviation rate for the minimizing measure with a
low-temperature marginal that puts mass $\kappa(1-\delta)$ in the shallow
well, normalized to the rate for $\alpha=1$ (symmetric potential); $\tau
_{1}=0.1,\ \tau_{2}=0.5$, $|\mathcal{S}|=12$.}%
\label{table:numFrantz2}%
\end{table}The results in Tables \ref{table:numFrantz}-\ref{table:numFrantz2}
illustrate how the rate function changes with $\kappa$ and $\delta$ when the
measures low-temperature marginal is restricted to assign less mass to
$[0,\infty)$ than the true invariant distribution $\bar{\mu}$. In particular,
it illustrates how the optimal value of the rate function decreases with the
level of asymmetry when the amount of mass that is redistributed is fixed.
This observation corresponds precisely to the large deviation interpretation
that for increased level of asymmetry the empirical measure will take a longer
time to converge. \newline\newline

\subsection{How poor sampling will occur}

Our second and final use of the rate function to study qualitative properties
of Monte Carlo addresses the following question. Suppose that a given Markov
process has invariant distribution $\mu$ that is concentrated in two wells (as
in a single temperature version of the model just considered). Suppose we also
consider a measure other than $\mu$ that one is likely to see prior to
convergence, e.g., a minimizer of the rate function subject to a constraint on
improperly assigning mass to the two wells (also as in the last section). By
solving the associated stochastic control problem we will find the change of
measure (change of jump rates) which minimizes the average cost per unit time
to hit the given target measure. Using the large deviation rate as in the
proof of Proposition \ref{prop:convTemp} to bound conditional probabilities
given a certain outcome, one can characterize the parts of the state space
where the observed empirical data collected along the simulated trajectory
deviates from what is expected based on the underlying dynamics. In other
words, if one were to attempt to infer the true dynamics based on the
empirical data, the solution to the control problem will tell us where these
inferred dynamics will deviate most from the true dynamics. One can imagine
that is at precisely these locations that greater accuracy in sampling (e.g.,
in approximating transition probabilities) would have the greatest impact on
overall performance. Although we do not propose a particular use along those
lines at this time, it seems to be interesting information with some potential
for improving schemes.

Here we consider the case of only one temperature $\tau=0.1$ and the
optimization problem
\[
\inf_{\nu\in\mathcal{P}(\mathcal{S})}\left\{  J(\nu):\nu([0,\infty
))=\kappa(1-\delta)\right\}  ,
\]
where $\kappa$ is the amount of mass the invariant distribution of the
underlying Glauber dynamics - with Franz potential - puts in the right well.
As in the previous subsection, the constraint amounts to placing less mass in
the shallow well compared to the invariant measure $\mu_{1}$.

Figure \ref{fig:value1} shows the value function $W$ that is associated with
the solution $\nu$ to the optimization problem for $\alpha=1$ (symmetric
well). From the value function we can compute the additional factor
$\text{exp}\{W(y)-W(x)\}$ in the optimal control for a jump from $x$ to $y$;
Figures \ref{fig:W1} and \ref{fig:W2} show the extra factor when $y$ is one
step to the right and left, respectively. Note that the controlled jump rates
differ from the uncontrolled ones in a neighborhood of the origin. Of course
these states are not visited much during the simulation of a trajectory, and
for this reason alone one might expect the numerics to poorly approximate the
true distributions (e.g., point to point conditional densities). However, this
statement applies to many parts of the state space, and the critical
difference is that errors here are more important in producing error to the
approximation of the invariant distribution. A possible fix would be to spend
some computational effort estimating the critical quantities (e.g., the
conditional probability to send at $\pm0.25$ after reaching $0.0$) beforehand,
and then use these more accurate estimates as the basis for a simulation
schemes that excises the corresponding parts of the simulated trajectory.
\begin{figure}
[ptb]
\begin{center}
\includegraphics[
height=1.785in,
width=2.4837in
]%
{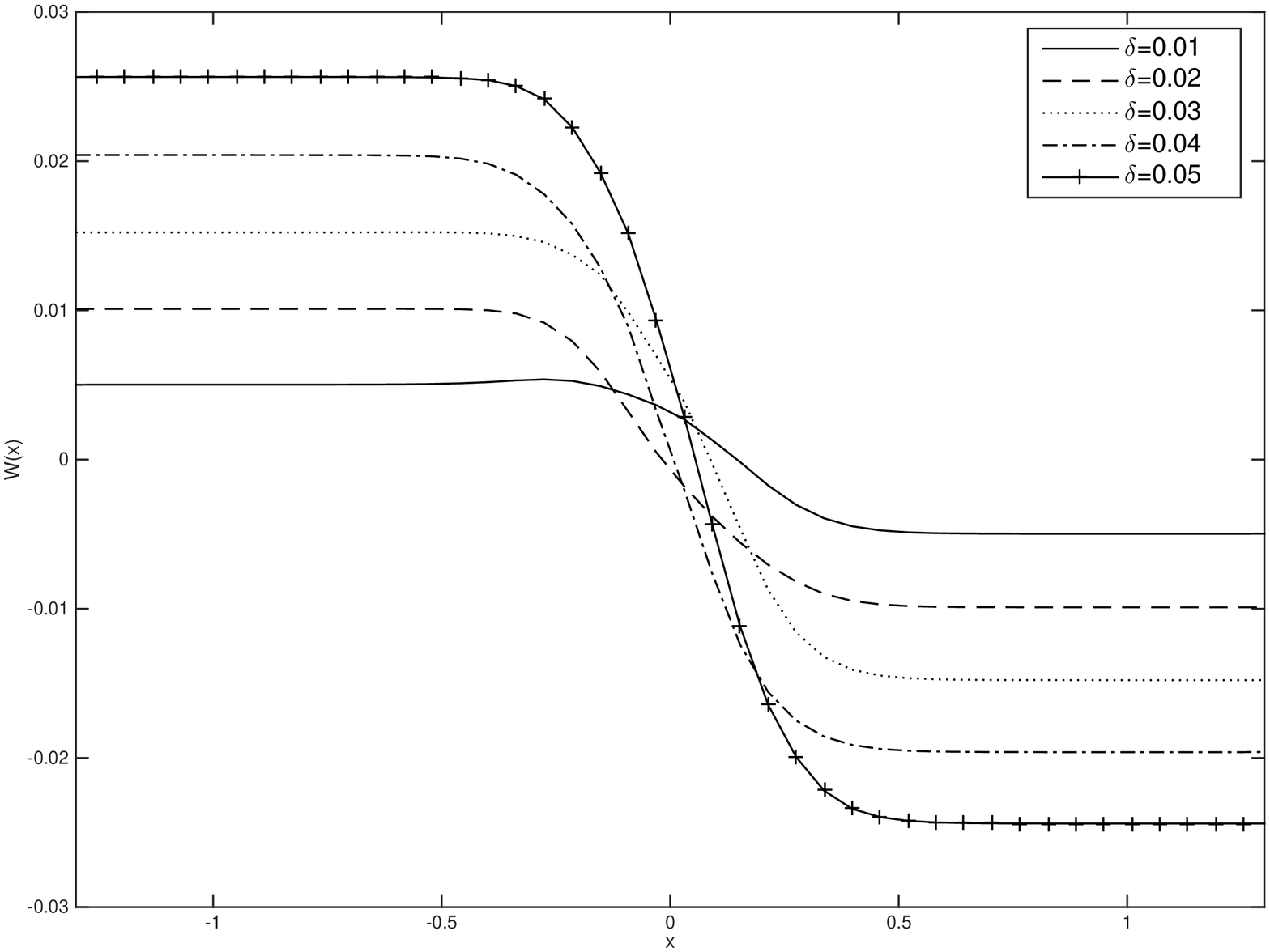}%
\caption{Value function $W$ when $\alpha=1$; $|\mathcal{S}|=50$, $\tau=0.1$. }%
\label{fig:value1}%
\end{center}
\end{figure}

\begin{figure}
[ptb]
\begin{center}
\includegraphics[
height=1.785in,
width=2.4837in
]%
{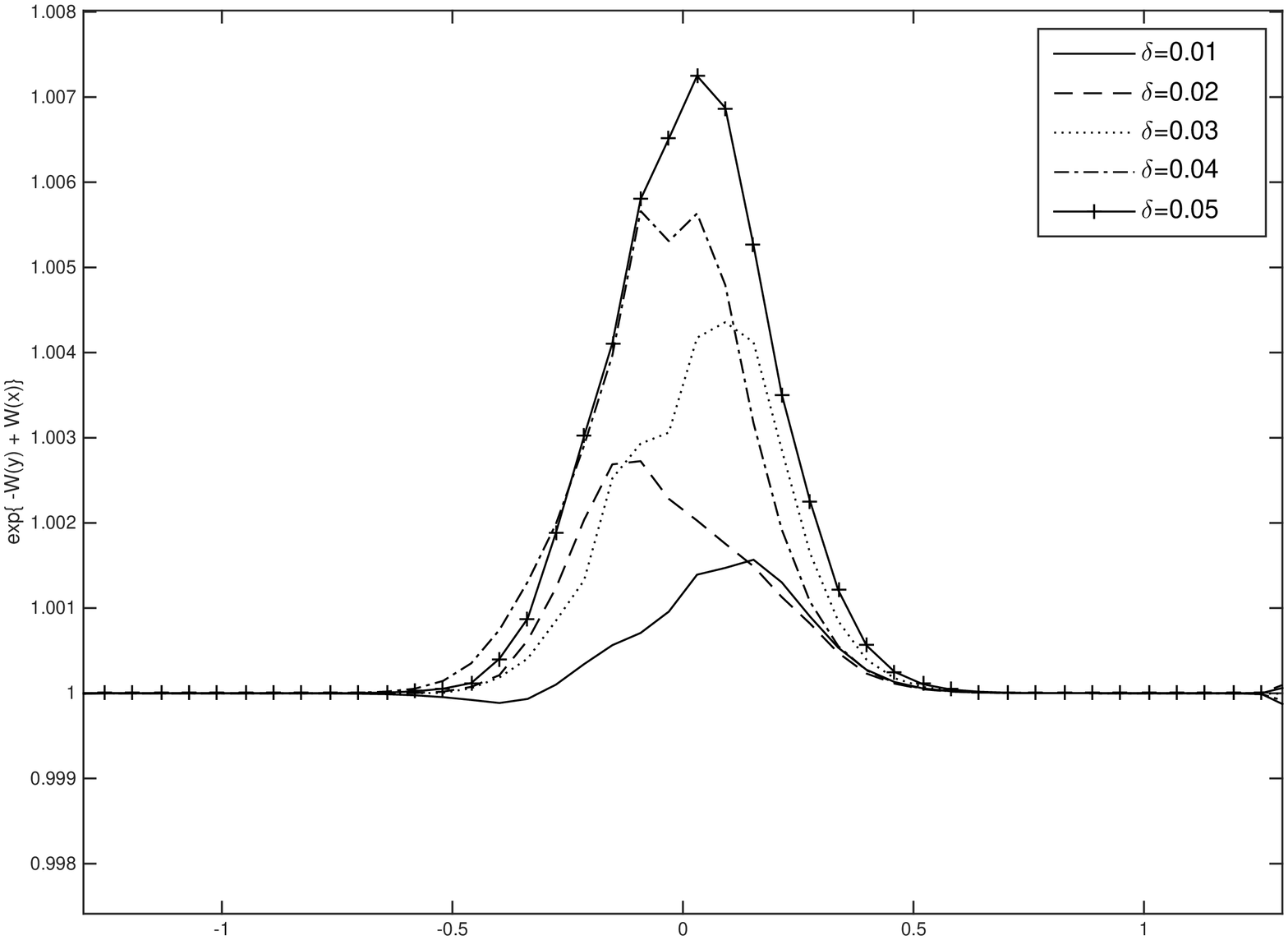}%
\caption{Additional factor in optimal control when $\alpha=1$ and jumps to the
right; $|\mathcal{S}|=50$, $\tau=0.1$. }%
\label{fig:W1}%
\end{center}
\end{figure}
\begin{figure}
[ptb]
\begin{center}
\includegraphics[
height=1.785in,
width=2.4837in
]%
{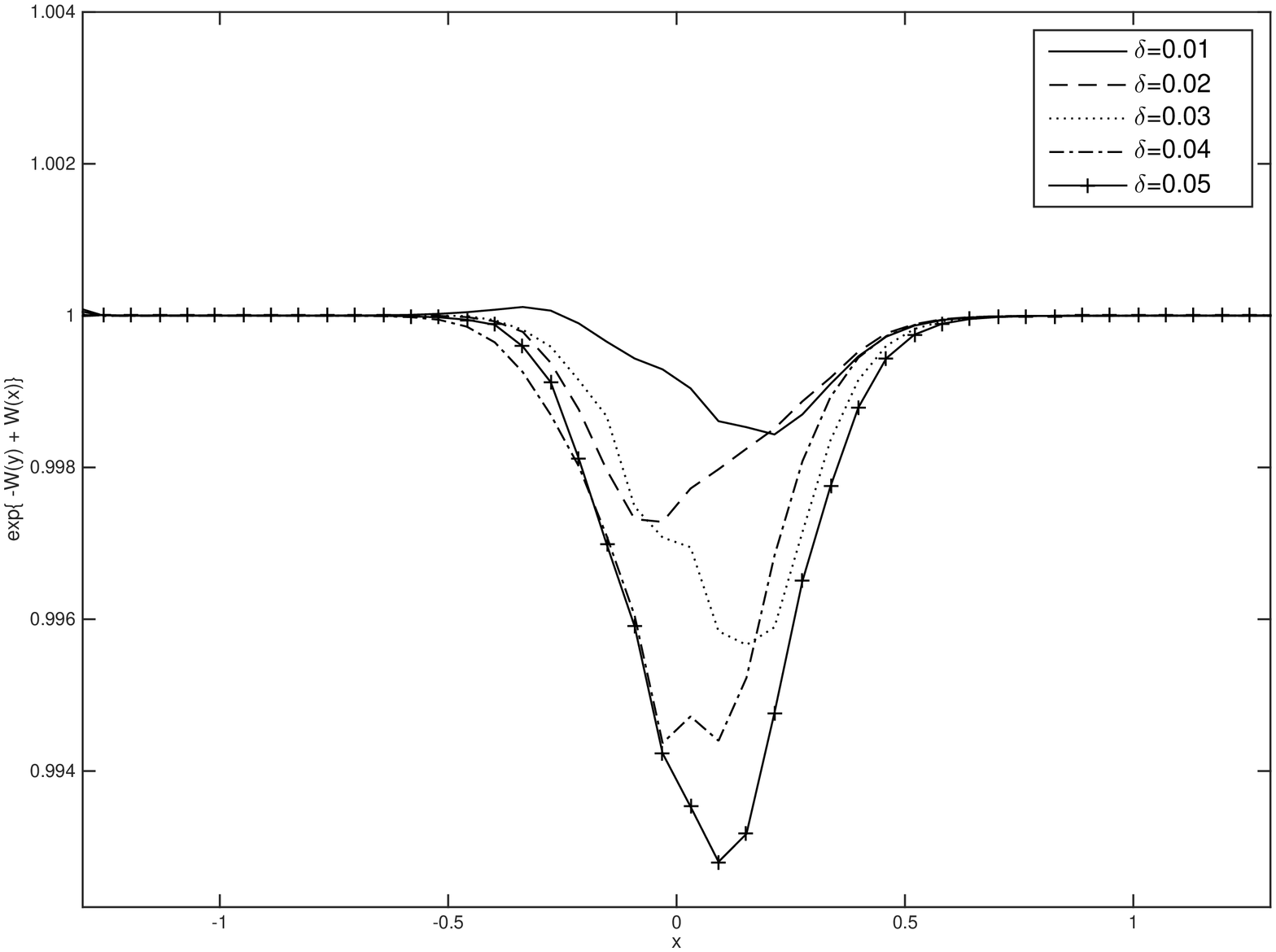}%
\caption{Additional factor in optimal control when $\alpha=1$ and jumps to the
left; $|\mathcal{S}|=50$, $\tau=0.1$. }%
\label{fig:W2}%
\end{center}
\end{figure}

\appendix

\section{Ancillary results}

\begin{lemma}
\label{lemma:ineqCont} For any sequence $a_{1},a_{2},\dots,a_{K}$ such that
$a_{i}\geq0$ for all $i$ and $K\in\lbrack0,\infty)$,
\[
\sum_{i=1}^{K}a_{i}=K
\]
and
\[
\sum_{i=1}^{K}a_{i}^{1/2}>K
\]
cannot both be true.
\end{lemma}

\begin{proof}
We can assume without loss that $K>0$. Let $b_{i}=a_{i}/K$, so that $\left\{
b_{i},i=1,\ldots,K\right\}  $ is a probability. By Hölder's inequality%
\[
\sum_{i=1}^{K}b_{i}^{1/2}\leq K^{1/2}\left(  \sum_{i=1}^{K}b_{i}\right)
^{1/2}=K^{1/2}.
\]
Using $b_{i}=a_{i}/K$ gives $\sum_{i=1}^{K}a_{i}^{1/2}\leq K$, which completes
the argument.
\end{proof}

\begin{lemma}
\label{lem:cost_assym}Consider the ergodic control problem or equivalent
minimization problem of Lemma \ref{lemma:equivOpt}, with $h(\mathbf{x}%
)=\lambda\rho(\mathbf{x})$. For two temperatures the optimal cost
$\gamma^{\ast}$ satisfies
\[
\gamma^{\ast}<\frac{\lambda}{2}.
\]
In the general case with $K$ temperatures, $\gamma^{\ast}<\lambda/K!$
\end{lemma}

\begin{proof}
To simplify the notation we consider the case $\lambda=1$. From Lemma
\ref{lemma:equivOpt} we know that there is a minimizing measure $\nu^{\ast}$
in \eqref{eq:optEq} and that the optimal cost $\gamma^{\ast}$ satisfies
\[
\gamma^{\ast}=J(\nu^{\ast})+\sum_{\mathbf{x}\in\mathcal{S}^{2}}\rho
(\mathbf{x})\nu^{\ast}(\mathbf{x}).
\]
Moreover, if $W^{\ast}$ is defined as
\[
W^{\ast}(\mathbf{x})=-\log\left[  \frac{d\nu^{\ast}}{d\bar{\mu}}\right]
^{1/2}(\mathbf{x}),\ \ \mathbf{x}\in\mathcal{S}^{2},
\]
then $(\gamma^{\ast},W^{\ast})$ is a solution to the Bellman equation \eqref{eq:HJeq2}.

Suppose that $W^{\ast}$ is a constant. Inserting this into the Bellman
equation yields, for each $\mathbf{x}\in\mathcal{S}^{2}$,
\[
0=-\gamma^{\ast}+\rho(\mathbf{x}),
\]
which cannot hold since $\rho$ is not a constant. Thus, $W^{\ast}$ cannot be a
constant function. This in turn implies that the likelihood ratio $[d\nu
^{\ast}/d\bar{\mu}]$ is not constant equal to $1$ (the only possible constant
value). Thus $\nu^{\ast}$ is not $\bar{\mu}$, the invariant measure for the
original symmetrized dynamics.

Inserting $\bar{\mu}$ into the objective function in \eqref{eq:optEq} gives
\[
J(\bar{\mu})+\sum_{\mathbf{x}\in\mathcal{S}^{2}}\rho(\mathbf{x})\bar{\mu
}(\mathbf{x})=\sum_{\mathbf{x}\in\mathcal{S}^{2}}\rho(\mathbf{x})\bar{\mu
}(\mathbf{x})=\frac{1}{2},
\]
where the second equality comes from $\rho(\mathbf{x})+\rho(\mathbf{x}%
^{R})=1=1/2+1/2$ and the symmetry of $\bar{\mu}$. Thus the cost associated
with the uncontrolled dynamics is $1/2$. Since $\nu^{\ast}$ is the unique
minimizer in \eqref{eq:optEq} and $\nu^{\ast}\neq\bar{\mu}$, it holds that
\[
\gamma^{\ast}=J(\nu^{\ast})+\sum_{\mathbf{x}\in\mathcal{S}^{2}}\rho
(\mathbf{x})\nu^{\ast}(\mathbf{x})<J(\bar{\mu})+\sum_{\mathbf{x}\in
\mathcal{S}^{2}}\rho(\mathbf{x})\bar{\mu}(\mathbf{x})=\frac{1}{2}.
\]

The argument for $K>2$ temperatures is completely analogous.
\end{proof}

\begin{lemma}
\label{lem:strict_convex} Assume $S$ is a finite set and that $\Gamma_{x,y}$,
$x,y\in S$ is the intensity matrix of an ergodic Markov chain on $S$ with
invariant probability distribution $\bar{\mu}$. Let $q(x)=\sum_{y\in S}%
\Gamma_{x,y}$, and for $\nu\in\mathcal{P}(S)$ with $\theta(x)=\nu(x)/\bar{\mu
}(x)$ let
\[
J(\nu)=\sum_{x\in S}q(x)\theta(x)\bar{\mu}(x)-\sum_{x,y\in S}\theta
^{1/2}(x)\theta^{1/2}(y)\Gamma_{x,y}\bar{\mu}(x).
\]
Then $J(\nu)$ is strictly convex on the relative interior of $\mathcal{P}(S)$.
\end{lemma}

\begin{proof}
It is enough to show the strict convexity of%
\[
\theta(\cdot)\rightarrow-\sum_{x,y\in S}\theta^{1/2}(x)\theta^{1/2}%
(y)\Gamma_{x,y}\bar{\mu}(x)
\]
for $\theta(x)\geq0$, $\sum_{x\in S}\theta(x)\bar{\mu}(x)=1$. Let
$\{x_{1},x_{2},\ldots,x_{K}\}$ be an enumeration of the distinct elements of
$S$, $\theta_{i}=\theta(x_{i}),\bar{\mu}_{i}=\bar{\mu}(x_{i})$ and
$f_{i,j}(\theta)=-\theta_{i}^{1/2}\theta_{j}^{1/2}$. If $M_{i,j}(\theta)$
denotes the matrix of second order partial derivatives of $f_{i,j}(\theta)$ at
$\theta$, then straightforward calculation shows that the eigenvalue $0$ is
repeated $K-1$ times, and $(\theta_{i}/\theta_{j}+\theta_{j}/\theta_{i})$ is
also an eigenvalue with eigenvector $\theta_{j}e_{i}-\theta_{i}e_{j}$. Hence
the null space of this matrix is the collection of vectors orthogonal to
$\theta_{j}e_{i}-\theta_{i}e_{j}$. Since $(\theta_{i}/\theta_{j}+\theta
_{j}/\theta_{i})>0$, $f_{i,j}(\theta)$ is strictly convex (as a function in
$\mathbb{R}^{K}$) at $\theta$ except in those directions orthogonal to
$\theta_{j}e_{i}-\theta_{i}e_{j}$.

Since $\Gamma_{x,y}$ is ergodic all states communicate, and so there exists a
sequence $1=i_{1},i_{2},\ldots,i_{K},i_{K+1}=1$ such that $\Gamma_{x_{i_{k}%
},x_{i_{k+1}}}>0$ for $k=1,\ldots,K$. Thus $-\sum_{x,y\in S}\theta
^{1/2}(x)\theta^{1/2}(y)\Gamma_{x,y}\bar{\mu}(x)$ is strictly convex except in
those directions that are orthogonal to each of $\theta_{i_{k+1}}e_{i_{k}%
}-\theta_{i_{k}}e_{i_{k+1}}$, which is exactly the set of directions spanned
by $(\theta_{1},\theta_{2},\ldots,\theta_{K})$. Since this direction cannot be
parallel to $\{\theta:\sum_{k=1}^{K}\theta_{k}\bar{\mu}_{k}=1\}$, $J(\nu)$ is
strictly convex on this set.
\end{proof}

\begin{remark}
\label{rmk:multtemp}The proofs in Section \ref{sec:PT} were largely confined
to the setting of two temperatures $\tau_{1},\tau_{2}$. This was to keep the
notation simple and the results generalize to any number $K\geq2$ of
temperatures. The only result which appears to substantially use that two
temperatures are considered is Lemma \ref{lemma:minPT}, and specifically the
argument by contradiction. Here we outline how the proof would proceed for the
general setting.

In the setting of $K$ temperatures the assumption \eqref{eq:condW} becomes
\begin{equation}
\sum_{\sigma\in\Sigma}e^{-2W(\mathbf{x}^{\sigma})}-K!=0,\ \forall\mathbf{x}%
\in\mathcal{S}^{K}. \label{eq:condW2}%
\end{equation}
We still have that $W(\mathbf{x})=0$ for $\mathbf{x}\in\mathcal{D}$.

Let $\mathcal{D}$ denote the set of diagonal states: $\mathcal{D}%
=\{\mathbf{x}:\mathbf{x}=\mathbf{x}^{\sigma}, \ \forall\sigma\in\Sigma_{K}\}$.
The only such states are those for which all components are equal. The cost
structure is such that $h(\mathbf{x})=1/K!$ for $\mathbf{x}\in\mathcal{D}$.

Consider the states that communicate directly with $\mathcal{D}$, i.e., those
only one step away from a diagonal state. Since the underlying processes only
jump one at a time there can only be a difference in one component, the others
remaining fixed. There are a total of $K!$ possible permutations in
$\Sigma_{K}$, $(K-1)!$ of which keep a specific component fixed. Thus, for a
state that is one step removed from the diagonal there are $(K-1)!$
permutations that result in the same state. Moreover, the diagonal state in
question will communicate directly with the remaining $K$ permutations as well.

The states one step away from a specific diagonal point can be viewed as
forming disjoint sets of states according to the previous description. For
each state $\mathbf{y}$ one step removed from an $\mathbf{x}$, there are $K$
distinct states $\mathbf{y}_{1},\dots,\mathbf{y}_{K}$ that are permutations of
$\mathbf{y}$ and communicate directly with $\mathbf{x}$. For each such
collection of states we can pick one to represent the collection (does not
matter which one we pick). Let $\mathcal{A}_{x}$ denote the collection of such
representative states $\mathbf{y}$. In the case of two temperatures this can
be phrased as only looking at states above the diagonal.

The Bellman equation for a diagonal state $\mathbf{x}$ takes the form
\[
0=\sum_{\mathbf{y}\in\mathcal{A}_{x}}\sum_{\sigma:\mathbf{y}^{\sigma}%
\neq\mathbf{y}}r(\mathbf{x},\mathbf{y})\left[  1-e^{-W(\mathbf{y}^{\sigma
})+W(\mathbf{x})}\right]  -\gamma+\frac{1}{K!}.
\]
The rates $r(\mathbf{x},\mathbf{y}^{\sigma})$ are all equal due to symmetry.
Combined with $W(\mathbf{x})=0$ for $\mathbf{x}\in\mathcal{D}$, this allows
the Bellman equation to be expressed as
\[
0=\sum_{\mathbf{y}\in\mathcal{A}_{x}}r(\mathbf{x},\mathbf{y})\left[
K-\sum_{\sigma:\mathbf{y}^{\sigma}\neq\mathbf{y}}e^{-W(\mathbf{y}^{\sigma}%
)}\right]  -\gamma+\frac{1}{K!}.
\]
Since $\gamma<(1/K!)$ and the rates are all non-negative it must be the case
that for at least one $\mathbf{y}\in\mathcal{A}_{x}$
\[
\sum_{\sigma:\mathbf{y}^{\sigma}\neq\mathbf{y}}e^{-W(\mathbf{y}^{\sigma}%
)}-K>0.
\]

For states one step from the diagonal, since $(K-1)!$ permutations will result
in the same state, the condition \eqref{eq:condW2} takes the form
\[
(K-1)!\sum_{\sigma:\mathbf{x}^{\sigma}\neq\mathbf{x}}e^{-W(\mathbf{x}^{\sigma
})}-K!=0\Leftrightarrow\sum_{\sigma:\mathbf{x}^{\sigma}\neq\mathbf{x}%
}e^{-W(\mathbf{x}^{\sigma})}-K=0.
\]
That is we need only be concerned with the permutations that switch the
location of the component that differs from the diagonal state (and the
$\sigma$ that corresponds to the identify map in $\Sigma_{K}$). There will
then be $(K-1)!$ permutations that produces the exact same state, yielding the
factor $(K-1)!$ in front of the sum.

For the reduced form of \eqref{eq:condW2} and the Bellman equation to hold, we
must have that
\[
\sum_{\sigma:\mathbf{x}^{\sigma}\neq\mathbf{x}}e^{-W(\mathbf{x}^{\sigma}%
)}-K=0,
\]
for all $\mathbf{y}$ that communicate with $\mathbf{x}$, and for at least one
such $\mathbf{y}$,
\[
\sum_{\sigma:\mathbf{y}^{\sigma}\neq\mathbf{y}}e^{-W(\mathbf{y}^{\sigma}%
)}-K>0.
\]
This is precisely the setting of Lemma \ref{lemma:ineqCont} with the $a_{i}$s
represented by $e^{-2W(\mathbf{y}^{\sigma})}$ for the $K$ relevant
permutations $\sigma$. The lemma then implies that \eqref{eq:condW2} is
inconsistent with the Bellman equation and therefore cannot hold. This
contradicts that $(M\bar{\nu})=\mu$.
\end{remark}

\bibliographystyle{plain}
\bibliography{main}

\end{document}